\documentclass[11pt, english]{article}
\font\smallit=cmti10

\usepackage{amssymb,amsthm,amsmath,amsfonts}
\usepackage[pdf]{pstricks}
\usepackage{pst-all, pst-3dplot, pstricks-add, pst-math, pst-xkey}
\usepackage{graphicx}
\usepackage{hyperref}
\usepackage{bbold}
\usepackage{enumerate}
\usepackage{epsfig}
\usepackage{subfigure}
\usepackage[capitalise, noabbrev]{cleveref}
	\crefname{subsection}{Subsection}{Subsections}
	\Crefname{subsection}{Subsection}{Subsections}
	
\usepackage{stmaryrd}
\usepackage{mathrsfs}

\usepackage[all,cmtip]{xy}

\newtheorem{theorem}{Theorem}

\newtheorem{corollary}[theorem]{Corollary}
\newtheorem{lemma}[theorem]{Lemma}

\newcommand{\A}{\ensuremath{\mathcal{A}}}
\newcommand{\R}{\mathbb{R}}
\newcommand{\Np}{\mathbb{Np}}
\newcommand{\E}{\mathbb{E}}
\newcommand{\M}{\mathbb{M}}
\newcommand{\Gs}{\mathbb{Gs}}
\newcommand{\D}{\mathbb{D}}
\newcommand{\s}{\ensuremath{\mathcal{S}}}
\newcommand{\U}{\ensuremath{\mathbb{U}}}
\newcommand{\GL}{{G^\mathcal{L}}}
\newcommand{\GR}{{G^\mathcal{R}}}
\newcommand{\HL}{H^\mathcal{L}}
\newcommand{\HR}{H^\mathcal{R}}
\newcommand{\XL}{{X^\mathcal{L}}}

\makeatletter
\newcommand{\oset}[3][0ex]{%
  \mathrel{\mathop{#3}\limits^{
    \vbox to#1{\kern-1\ex@
    \hbox{$\scriptstyle#2$}\vss}}}}
\makeatother
\makeatletter
\newcommand{\osett}[3][0ex]{%
  \mathrel{\mathop{#3}\limits^{
    \vbox to#1{\kern-1.1\ex@
    \hbox{$\scriptstyle#2$}\vss}}}}
\makeatother

\usepackage{mathabx}

\newcommand{\conjt}[1]{\osett{\curvearrowleftright}{#1}}
\newcommand{\conj}[1]{\oset{\curvearrowleftright}{ #1}}

\newcommand{\pura}[2]{\{\emptyset^{#1}\!\mid \!\emptyset^{#2}\}}

\newcommand{\cg}[2]{\left\{ #1\!\mid \!#2\right\}}
\newcommand{\cgs}[2]{\{ #1\!\mid \!#2\}}

\newcommand{\atom}[1]{\emptyset^{#1}}

\newcommand{\mup}{\curlywedgeuparrow}

\renewcommand{\ge}{\geqslant}
\renewcommand{\le}{\leqslant}
\newcommand{\su}{\succcurlyeq}
\newcommand{\pr}{\preccurlyeq}

\newcommand{\bs}[1]{\boldsymbol#1}
\newcommand{\abss}{absolute}
\newcommand{\Abss}{Absolute}
\newcommand{\abs}{\abss\ }
\newcommand{\Abs}{\Abss\ }
\renewcommand\emptyset \varnothing

\newcommand\fuzzyU[1]{\lower0.35ex\hbox{${\ \mathrel{\rotatebox{90}{$=$}}}_{#1}\,$}}

\theoremstyle{definition}
\newtheorem{definition}[theorem]{Definition}

\newtheorem{observation}[theorem]{Observation}
\newtheorem{remark}[theorem]{Remark}
\newtheorem{example}[theorem]{Example}

\begin{document}

\begin{center}
\uppercase{\bf Absolute Combinatorial Game Theory}
\vskip 20pt
{\bf Urban Larsson\footnote{Partially supported by the Killam Trusts; urban031@gmail.com}}\\
{\smallit School of Computing, National University of Singpore, Singapore}\\
{\bf Richard J.~Nowakowski\footnote{r.nowakowski@dal.ca}}\\
{\smallit Department of Mathematics and Statistics, Dalhousie University, Canada}\\
{\bf Carlos P. Santos\footnote{Centro de Estruturas Lineares e Combinatorias, Faculdade de Ciencias, Universidade de Lisboa, Campo Grande, Edificio C6, Piso 2, 1749-016 Lisboa, Portugal; cfmsantos@fc.ul.pt}}\\
{\smallit Center for Linear Structures and Combinatorics, Portugal}\\
\end{center}

\begin{abstract}
We propose a unifying additive theory for standard conventions in  Combinatorial Game Theory, including normal-, mis\`ere- and scoring-play, studied by Berlekamp, Conway, Dorbec, Ettinger, Guy, Larsson, Milley, Neto, Nowakowski, Renault, Santos, Siegel, Sopena, Stewart (1976-2019), and others. A game {\em universe} is a set of games that satisfies some standard closure properties. Here, we reveal when the fundamental game comparison problem, ``Is $G\su H$?'', simplifies to a constructive `local' solution, which generalizes Conway's foundational  result in ONAG (1976) for normal-play games. This happens in a broad and general fashion whenever a given game universe is {\em absolute}. Games in an absolute universe satisfy two properties, dubbed {\em parentality} and {\em saturation}, and we prove that the latter is implied by the former.  Parentality means that any pair of non-empty finite sets of games is admissible as options, and saturation means that, given any game, the first player can be favored in a disjunctive sum. Game comparison is at the core of combinatorial game theory, and for example efficiency of potential reduction theorems rely on a local comparison. We distinguish between three levels of game comparison; superordinate (global), basic (semi-constructive) and subordinate (local) comparison. In proofs, a sometimes tedious challenge faces a researcher in CGT: in order to disprove an inequality, an explicit distinguishing game might be required. Here, we explain how this job becomes obsolete whenever a universe is absolute. Namely, it suffices to see if a pair of games satisfies a certain Proviso together with a Maintenance of an inequality.
\end{abstract}

\section{Introduction}
A combinatorial game is a (terminating) two-player zero-sum game of perfect information and no randomness. Given a game and a starting player, the two players alternate moving, until the current player has no options, and a winner is declared. 
Amongst the most popular conventions, we find \textit{normal-play}, a player who cannot move loses; and \textit{mis\`ere-play}, this player wins; these two conventions will henceforth be called the {\em classical conventions}. The third main track is \textit{scoring-play}, a convention with real valued `scores' attached to the terminal positions, and the two players, Left and Right, seek to maximize and minimize this score, respectively. Combinatorial game theory concerns optimal play, captured in so-called {\em outcome functions}, and traditionally they have distinct definitions, depending on the convention for hand. Likewise, the notion of a {\em disjunctive sum} of games and the definition of {\em game comparison}, used to require slightly different definitions in classical vs. scoring settings. This study concerns a unification of terminating games in an {\em absolute theory} that combines the best of all, and more. The absolute theory provides recursive solutions for game comparison, a notion that is non-constructive by definition. Recursive solutions were known before in the literature, but they were always treated individually; here we unify the approach.

The two players are usually called \textit{Left} (female) and \textit{Right} (male),\footnote{After Louise and Richard Guy, who contributed so much to this field.}. We distinguish between the set of Left-options, those positions that Left can reach in one move, and similarly the set of Right-options,
 denoted by $\GL$ and $\GR$ respectively. Any game $G$ can be represented by two such sets and we use the standard recursive definition of a game, $G=\cgs{\GL}{\GR}$. The games are commonly represented by a rooted (finite or infinite) tree, with two types of edges (Left and Right slanting respectively), called the `game tree'. The nodes are positions that can be reached during play, the edges represent the possible moves, and the root is the current position. The children of a node are all the positions that can be reached in one move and these are called the \textit{options}. 

We restrict attention to those game trees for which each node, a.k.a {\em follower} has a finite number of edges (options) and where any, not necessarily alternating, play sequence  is finite; this is sometimes called the class of {\em short} combinatorial games. Thus, any such game can be expanded in terms of elements of its \emph{terminal} positions. This gives the common proof technique `induction on the options' since the depth of the game tree of an option is at least one less than that of the original position.

The reader is invited to consult the standard text books in combinatorial game theory:  \cite{AlberNW2007,BerleCG2001,Conwa2001,Siege2013}. Although they do not cover scoring-play, they lay the foundation for the field, and in particular through their thorough development of the normal-play theory. Here, we will demonstrate that normal-play ideas are at the core of many other conventions, in particular mis\`ere- and scoring-play, and where suitable restrictions may be imposed on the games within a convention.

Let us recall the foundations of the field. Given a classical convention, there are four {\em outcome} classes. The starting player, Left or Right, wins a game $G$ if the outcome $o(G) = \mathscr N$, the second player wins if $o(G) = \mathscr P$, player Left wins independently of who starts if $o(G) = \mathscr L$, and player Right wins independently of who starts if $o(G) = \mathscr R$. The first fundamental theorem of combinatorial games states that these definitions induce a tetra partition of games.

\begin{theorem}[First Fundamental Theorem, folklore]\label{thm:main_normal}
If $G$ is a classical game, then $o(G)\in \{\mathscr N,\mathscr P,\mathscr L,\mathscr R\}$.
\end{theorem}

These outcomes inherit a partial order from a given total order of the possible results: `Left wins' $={\rm L}>{\rm R}=$ `Right wins'.  We get the  order between outcomes as $\mathscr L>\mathscr N>\mathscr R$, and $\mathscr L>\mathscr P>\mathscr R$, but where $\mathscr N$ and $\mathscr P$ are incomparable because the winner depends on who starts.

In scoring-play the situation is more varied and seemingly more challenging; instead of a tetra partition of outcomes, outcomes are represented by  ordered pairs of real numbers, thus an uncountable number of outcome classes. The purpose of this work is to demonstrate how to overcome this apparent complication, and show that, in a fundamental sense, the similarities are profounder than the differences. Formal definitions suitable for absolute study will be the content of Section~\ref{sec:def}, and onwards, while this section and the next one aims to provide a sufficient intuition to motivate a thorough study of the subject.

A {\em universe} of combinatorial games is a set of games that satisfies standard closure properties: closure under disjunctive sum, closure under taking options, closure under conjugate, and it has a neutral element in `the empty game' (Section~\ref{sec:univ}).

The game addition operator `+', called {\em disjunctive sum}, means that the current player plays in exactly one game component out of a finite sum of games (Section~\ref{sec:gameform}). For any specified universe, this induces a partially ordered monoid of games, with the neutral element the empty game, henceforth denoted ${\bs 0}$. 

\begin{definition}[Partial Order]\label{def:partord}
Consider two games $G$ and $H$ in a specified universe $\U$. Then $G\su_\U H$ if, for all games $X$ in this universe, $o(G+X)\su o(H+X)$. If $G\su_\U H$ and $H\su_\U G$, then $G=_\U H$.\footnote{We use ``$\su$'' for partial order of games.} 
\end{definition}

The subscript \U\ is omitted when the universe is understood. This is the standard definition for order of combinatorial games. Its motivation stems from games like the oriental game of {\sc go}, where the game decomposes into independent components during play. Are two given game components comparable in {\em any} possible play situation? In this sense,  Definition~\ref{def:partord} leaves no ambiguity of preference. 

A note on terminology. We often call the games $X$ in Definition~\ref{def:partord} by {\em distinguishing} games, in the sense: if there is a distinguishing game, then the inequality (equality) does not hold, and otherwise, if there is no game $X$ to distinguish the games $G$ and $H$, then the inequality (equality) holds.

But, given this seemingly overwhelming requirement for games to be comparable, one might guess that for most universes, all games would be incomparable. When, if ever, could one hope for an interesting structure in the induced partial order of games?

\subsection{The appeal of normal-play, and the second fundamental theorem}
One of the most beautiful discoveries of combinatorial game theory is that, in normal-play, the maximizer, Left, wins playing second in the game $G$ if and only if $G\su {\bs 0}$. Moreover, it is well known that normal-play games constitute a group structure. This leads to a \textit{local}, and therefore {\em constructive}, game comparison: namely, it suffices to understand $G$ and $H$ in terms of their followers.\footnote{We will use the terms `local' and `constructive' interchangeably depending on which aspect we wish to emphasize. Other terms in the literature with the same meaning is ``option-only'', ``play-comparison'', ``algorithmic'' and ``recursive''.} Specifically, $G\su H$ if and only if Left wins the game $G - H=G+(-H)$ playing second. Here, the negative of a game is the game where the players have swapped roles. 

There are three equivalent ways of expressing this result. A typical Left (Right) option of a game $G$, is denoted $G^L$ ($G^R$).
\begin{theorem}[Second Fundamental Theorem, \cite{Conwa1976a, BerleCG1982a}]\label{thm:2FT}
Consider a pair of normal-play games $G,H$. The following items are equivalent:
\begin{enumerate}
\item $G\su H$.
\item Left wins $G-H$ playing second.
\item No $G^R\pr H$ and no $H^L\su G$
\item \begin{enumerate}
  \item For all $G^R$, there is an $H^R$ such that $G^R\su H^R$ or there is a $G^{RL}$ such that $G^{RL}\su H$.
  \item For all $H^L$, there is a $G^L$ such that $G^L\su H^L$ or there is an $H^{LR}$ such that $G\su H^{LR}$.
\end{enumerate}
\end{enumerate}
\end{theorem}

The fourth item is of special interest to this work, and although intuitively clear, we will include a proof on page~\pageref{lem:main}.
In combination with the second item, the interpretation is a `maintenance' of game preference via the possible Left responses to Right's opening move, either $G^R$ or $-H^L$. In the pursuit of a generalization of the second fundamental theorem, the main task is to show that an analogous reasoning holds in a much more general context, where group structure cannot be assumed.

\subsection{The fundamental theorem of absolute universes}\label{sec:basic}
 In this study, we propose an extension to the second fundamental theorem. To this purpose, it will be necessary to decompose the outcome function in its two parts, depending on who starts. We write $o(G)=(o_L(G),o_R(G))$, where $o_L(G)$ ($o_R(G)$) denotes the Left-outcome (Right-outcome): the result in optimal play when Left (Right) starts. The formal definition of outcome will require some preliminary machinery (see Section~\ref{sec:def}), but the main idea is transparent by this notion, and it suffices for this section and Section~\ref{sec:motiv}. 
 
 A {\em Left-atomic} ({\em Right-atomic}) game is a game where Left (Right) cannot move, and a game is {\em atomic} if it is Left-atomic or Right-atomic. A game is {\em purely atomic} if it is both Left- and Right-atomic.

 By combining the fourth item in Theorem~\ref{thm:2FT}, henceforth called the Maintenance (Common Normal Part in \cite{LarssNS2018}) with a Proviso that concerns addition with atomic games, we obtain a unifying theory for many short combinatorial games.  
 
 A universe of combinatorial games is {\em absolute} if it is {\em parental} (Definition~\ref{def:parental}) and {\em outcome saturated} (Definition~\ref{def:dense}). In essence, parental means that any pair of non-empty sets of game forms of the universe makes a game form of the universe, and saturated means that given a game $G$, the there is a game $H$ such that the first player is favoured in the disjunctive sum $G+H$.  

\begin{theorem}[Absolute Fundamental Theorem]\label{thm:basic}

Consider a pair of games $G,H$ in an absolute universe, $\mathbb U$. The relationship $G\succcurlyeq H$ holds if and only if the following two conditions are satisfied,\\

\noindent
Proviso:
\begin{enumerate}[]
 \item $o_L(G+X)\geqslant o_L(H+X)$ for all Left-atomic $X\in \U$;

 \item
 $o_R(G+X)\geqslant o_R(H+X)$ for all Right-atomic $X\in \U$;
\end{enumerate}
\noindent
Maintenance:
\begin{enumerate}[]
\item For all $G^R$, there is an $H^R$ such that $G^R\succcurlyeq H^R$, or there is a $G^{RL}$ such that $G^{RL}\succcurlyeq H$;
\item For all $H^L$, there is a $G^L$ such that $G^L\succcurlyeq H^L$, or there is an $H^{LR}$ such that $G\succcurlyeq H^{LR}$.
\end{enumerate}
\end{theorem}

The Proviso adjusts for a terminal situation that is more complicated than the normal-play ending.  It is not yet a purely local consideration. But in all the universes we have encountered, it can be simplified, which is the content of Section~\ref{sec:subordinate}. The Maintenance of \cref{thm:basic} clearly replaces the infinite number of comparisons by a local condition.

We propose a taxonomy of `game comparison'.\footnote{
Taxonomy references: 
Croft, William A. \& D.A. Cruse (2004). Cognitive Linguistics. Cambridge, Cambridge University Press. Ungerer, Friedrich \& Hans-Jörg Schmid (1996). An Introduction to Cognitive Linguistics. London, Longman.}
\begin{itemize}
\item {\em Superordinate}: This is the standard non-constructive game comparison by Definition~\ref{def:partord};
\item {\em Basic}: Absolute universes have a semi-constructive game comparison, as detailed in Theorem~\ref{thm:basic}.
\item {\em Subordinate}: Many absolute universes have local/constructive comparison. Here we find, for example, the Second Fundamental Theorem, which concerns normal-play, a special lucky case, where the Proviso returns an empty condition.
\end{itemize}

Before we move on let us display relevant literature in a table format.
\begin{table}[hbt]
\caption{Important developments in the theory of additive combinatorial games. Results published since 2007 are in bold. Cyclic (or loopy) games are not (yet) included in this framework. Section~\ref{sec:motiv} gives a detailed but still informal account of how some of these game universes interact with the absolute theory.}
\begin{center}
\begin{tabular}{c|l|l|l|}
&\multicolumn{2}{c}{Play Convention}\\
\hline
Class&Normal&Mis\`ere&Scoring\\
\hline
Partizan&\cite{BerleCG1982a,Conwa1976a}&\hspace{.8cm}\textbf{\cite{MesdaO2007,Siege2015b}}&\textbf{\cite{LarssNNS2016,LarssNS2018c,Stewa2011}}\\
Dead-ending&\cite{BerleCG1982a,Conwa1976a}&\hspace{.8cm}\textbf{\cite{Mille2015, MilleR2013a, LarssMNRS2018}}&\\
Dicot/All-small&\cite{Conwa1976a},\hspace{.4cm}\textbf{\cite{McKay2011}}&\hspace{.8cm}\textbf{\cite{DorbeRSS2015}} &\cite{Milno1953, Ettin1996, Ettin2000}\\
Impartial&\cite{Grund1939, Sprag1935}&\cite{GrundS1956},\,\textbf{\cite{Plamb2005,Plamb2009,PlambS2008}}&\\
Loopy&\cite{Conwa1978a},\cite{FraenY1986}\hspace{.4cm}\textbf{\cite{Siege2009,Siege2009a}}&&\\
\hline
\end{tabular}
\end{center}
\label{table:theories}
\end{table}%

\subsection{Outline}
In Section~\ref{sec:motiv}, we recall some (absolute) universes from the literature, that motivated this work; this section can be omitted for a reader who already feels sufficiently motivated, and prefers to dive directly into the more technical definitions of absolute play, required in the proofs to come.  The forthcoming sections do not depend on Section~\ref{sec:motiv}. 

Section~\ref{sec:def} conceptualizes \abs combinatorial game theory, which leads to a partial order for universes at a superordinate level of taxonomies; that is we convert Definition~\ref{def:partord} to our setting. Section~\ref{sec:downlinked} concerns a proof technique via so-called {\em downlinked} games \cite{Siege2013}. In Section~\ref{sec:main}, we prove the main result, Theorem~\ref{thm:basic}, which promises a semi-constructive game comparison in \abs universes, via the Proviso; thus this section deals with a basic taxonomic level. In Section~\ref{sec:subordinate}, we discuss the translation of the Proviso in some familiar universes of games, corresponding to a subordinate and constructive level of game comparison.


\section{A motivation via standard universes of games}\label{sec:motiv}

This section motivates the absolute framework as presented in Section~\ref{sec:def} and onwards.  The terminology and notation introduced here is local, i.e. not yet unifying, and the rest of the paper does not depend on it. 

Independently of the winning convention, a reader may be acquainted  with other restrictions on the set of permitted games.  For example, a game is \textit{impartial}, if both players have the same options, and each option is impartial. Note that, because of the parentality condition, as outlined in Section~\ref{sec:basic}, before Theorem~\ref{thm:basic}, this restriction gives a too small universe of games. On the other hand impartial games have a more rudimental partial order in the sense that all games are either equal or incomparable; thus absolute theory is not needed, but other methods are required, such as the famous mis\'ere quotients \cite{Plamb2005,Plamb2009,PlambS2008}. 

In order to illuminate Theorem~\ref{thm:basic} further, in Section~\ref{sec:dicot} (Example~\ref{ex:dicot}) we begin by comparing some simple games in a well known absolute universe, namely the mis\`ere {\em dicots}, here denoted $\D$,  e.g. \cite{DorbeRSS2015}.\footnote{The authors prove uniqueness of canonical forms, but they omit to deduce the affiliated local game comparison.} Dicot means that if one of the players has an option, then so does the other player, and each option is a dicot. 

Dicot universes have nice properties, but they do not capture the richness of absolute theory. The {\em dead-ending} universe (Sections~\ref{sec:waiting} and \ref{sec:concretewaiting}) and the {\em guaranteed universe} (Section~\ref{sec:guaranteed}) give appealing motivations for the theory (with formal treatments in Section~\ref{sec:subordinate}). Here they are used to explain the birth of the absolute ideas, and the similarities and differences in the classical respectively scoring settings. 
A game is \textit{dead-ending}, if, whenever one of the players runs out of options, then this player will never again be able to move in this game (regarded as a component in a disjunctive sum of games). The universe of dead-ending games will be denoted $\mathbb E$. 
Scoring games have scores attached to the terminal situations, and it is possible to define a {\em guaranteed} property where a player who runs out of options in one component can never score worse in this component, if the disjunctive sum play continues. This property will define the  universe of guaranteed scoring games, $\Gs$. A preliminary observation here is that both these restrictions limit what can happen after a terminating situation for one of the players, and similar to normal-play, this player would typically not be favoured in the full disjunctive sum, by running out of moves in one of the components.

\subsection{A dicot universe}\label{sec:dicot}
The dicot restriction leads to a tremendous simplification of the Proviso, namely it reduces to $o(G)\su_\D o(H)$, because, in $\D$, the only possibility for an atomic game $X$ is $X={\bs 0}$, the empty game. 

So, why is the Proviso necessary? We let $*=\cg{{\bs 0}}{{\bs 0}}$, and $*2=\cg{\bs 0,*}{\bs 0, *}$ be literal forms.  The games $G=\cg{*2}{*2}$ and $H=0$ maintain inequality modulo dicot mis\`ere, but $G\not\su_\D H$, because $o_L(G)=_\D={\rm R}$, but $o_L(H)=_\D {\rm L}$, which indeed is captured by the Proviso.

Observe  that Maintenance in normal-play does not imply Maintenance in an absolute universe. For example, take the literal forms $G=\cg{\bs 0}{\cg{*2}{*2}}$ and $H=*$. Then $G\su_\Np H$, with $\Np$ the normal-play universe. But $G,H$ does not satisfy Maintenance modulo dicot mis\`ere.

Let us view some more examples of dicot mis\`ere-play.

\begin{example}[Theorem~\ref{thm:basic}, Dicot Mis\`ere]\label{ex:dicot}
In both normal- and mis\`ere-play, define the literal form games $*=\cg{\bs 0} {\bs 0}$, $\uparrow \; = \{\bs 0\mid *\}$ and $\mup \;=  \{\bs 0,*\mid *\}$.  Observe that, by standard reversibility, in normal-play $\mup\; =_\Np\; \uparrow$. However, in dicot mis\`ere we have that $\mup\; \succ_\D\; \uparrow$. To prove this, observe that the Maintenance holds (with say $G=\;\mup$ and $H=\;\uparrow$), and ${\rm L}=o_L(\mup)>_\D o_L(\uparrow)={\rm R}$. Thus, we get $\mup\;\su_\D\; \uparrow$ but $\uparrow\; \not\su_{\mathbb D}\; \mup$.  For another example, in dicot mis\`ere, $\uparrow \fuzzyU{\D}\; {\bs 0}$. This follows by noting that in normal-play $\uparrow\; \succ_{\Np} {\bs 0}$, so we cannot have ``$\pr$'', but in dicot mis\`ere, ${\rm R}=o_R(\uparrow)<_{\mathbb D} o_R({\bs 0})={\rm L}$.
\end{example}

Thus Theorem~\ref{thm:basic} simplifies game comparison a lot, because it is no longer required to find a distinguishing game `$X$' to separate games (see e.g. \cite{DorbeRSS2015} where that technique has been used for dicot mis\`ere).

This signals the usefulness of the absolute theorem. However, while the intuition of the Maintenance should be clear, it might take a bit longer to capture the essence of the Proviso. The Proviso has the flavour of a certain {\em waiting problem}, which is connected to the machinery of atomic games, and the universe of so-called {\em dead-ending} mis\`ere-play games \cite{MilleR2013a} serves a perfect fit.

\subsection{No waiting at a dead-end}\label{sec:waiting}
Dead-ending gives a nice structure to mis\'ere-play. Suppose that Left has no options in the game $G$ (i.e. $G$ is Left-atomic), but when played together with another game $H$, written $G+H$, then Left can play to say $G+H^L$. Now, if Right plays in the game $G$, to say $G^R+H^L$, then, possibly Left could play in the `$G$' component to open up new possibilities, unless $G$ is dead-ending. More specifically, suppose that, to win, Left must {\em wait} for Right to play first in some sub-position of the `$H$' component. This could be accomplished if Left, at this point, had benn awarded some strategic moves in the $G$ component. However, if $G$ is dead-ending, then this will never happen; Left can never hope to play there. Thus, we see an example of how a restriction of the set of games may simplify analysis of games.

The {\em waiting problem} concerns such situations, where, say Left, by waiting in $H$, can hope for a better future through opened possibilities in the $G$ component. If such threats do not exist, then Right can ignore the $G$ component and instead focus all efforts on alternating play in the $H$ component, or if he, at some point desires to shift the order of play in $H$, he may play a move in $G$.

Let us illustrate this important idea, that inspired this work, by a more concrete example. If $G$ is Left-atomic (Right-atomic) we typically write $G=\cg{\emptyset}{\GR}$ ($G=\cg{\GL}{\emptyset}$).\footnote{Later, in Section~\ref{sec:gameform}, inspired by guaranteed scoring-play, we will adorn the empty sets of options, to allow for absolute generalizations.}

\subsection{A concrete example of the waiting problem}\label{sec:concretewaiting}
Milley and Renault \cite{MilleR2013a} first noted the importance of the waiting problem via Left- and Right-atomic games in pairwise  comparison of dead-ending games. The restrictions before this were dicot universes, and Milnor \cite{Milno1953} was the first to observe that the  restriction of the dicot scoring-play universe, where it is never bad to play first, allows for local game comparison, similar to normal-play.\footnote{In fact today we know that Milnor's universe is the same as a subset of the normal-play games, namely the reduced canonical form of normal-play games, games without infinitesimals.} But, similarly to the setting for impartial games, Milnor's setting is too restricted to fit within the absolute umbrella.

Consider the literal form game $G=\cg{\emptyset}{\cg{\bs 0}{\emptyset}}$ (Figure~\ref{fig:gametree}) in a mis\`ere-play disjunctive sum with the game $*$, written $G+*$, and note, first of all,  that $G$ is not dead-ending. If the game $G$ is played alone, Left wins, playing first, but she loses if Right starts: this is a `hot Left-atomic game', where both players want to start. But, on the other hand, the empty set of options $\GL=\emptyset$ is not good for Left, played in the context of the disjunctive sum $G+*$. Namely, playing first, Left has to move to $G$, and then Right moves to the game $\cg{\bs 0}{\emptyset}$, where Left loses. The presence of $*$ allowed Right to {\em wait} for his opportunity to play in $G$, where Left's play is bound to lose.

 \begin{figure}[ht!]
 \begin{center}
 \psset{xunit=1.0cm,yunit=1.0cm,algebraic=true,dotstyle=o,dotsize=3pt 0,linewidth=0.8pt,arrowsize=3pt 2,arrowinset=0.25}
 \begin{pspicture*}(-7,0)(10,5)
 \psline(-1.54,3.14)(-0.68,1.98)
 \psline(-0.68,1.98)(-1.56,0.82)
 \rput[tl](-1.7,3.8){$G$}
 \rput[tl](-0.5,2){$\cg{\bs 0}{\emptyset}$}
  \begin{scriptsize}
 \psdots[dotstyle=*](-1.54,3.14)
 \psdots[dotstyle=*](-0.68,1.98)
 \psdots[dotstyle=*](-1.56,0.82)
 \end{scriptsize}
 \end{pspicture*}
 \caption{The game $G=\cg{\emptyset}{\cg{\bs 0}{\emptyset}}$.}\label{fig:gametree}
 \end{center}

 \end{figure}

Observe, that `hot-atomic' games do not exist in normal-play. In normal-play the worst thing that can happen, in any situation, is to run out of move options; it is well-known \cite{Conwa2001} that this cannot be good, in any context. This is not just a subtle difference between conventions. Indeed, in normal-play a large number of moves is never bad for a player, but in mis\`ere-play this is usually bad (but not always).\footnote{Normal-play has a unique so-called {\em zugzwang} (where no player wants to start), namely the neutral element 0, but mis\`ere-play has infinitely many zugzwangs, and in particular each game of the form $\cg{\overline{\ell}}{\overline{r}}$ is a zugzwang, where $\overline{\ell}$ is $\ell$ moves for Left, and $\overline{r}$ is $r$ moves for Right. Consider the game $\overline {100}$. Thus Left has 100 moves, whereas Right cannot move. This seems like a terrible game for Left, in the mis\`ere-play convention, but note that played together with the huge zugzwang $\cg{\overline {1000}}{\overline{-1000}}$, then Left wins the game $\overline {100}+\cg{\overline {1000}}{\overline{-1000}}$, because Right will have to open the huge zuzwang, no matter who starts. However, since in normal-play, the only zuzwang is 0, then this type of situation cannot happen. See the recent manuscript \cite{LarssMNRS2018}, for more on this topic, and an introduction to the relevant concept of a `perfect murder', in the context of dead-ending games.}

This discussion illustrates a usual difficulty in CGT proofs. Milley and Renault's dead-ending restriction achieves some interesting results because, as we pointed out, the potential power of waiting for new opportunities is dismantled in their universe. Indeed, if $G$ in Figure~\ref{fig:gametree} were instead the dead-ending $\cg{\varnothing}{\bs 0}$, then Right's potential ``waiting'' in $G+*$ would be useless, since it would lead to  swift Left win. 
Let us give one more example on the waiting problem, which illustrates how dead-ending mis\`ere differs from the full universe of mis\`ere-play, denoted $\M$, by using explicitly the Proviso from Theorem~\ref{thm:basic}.

\begin{example}[Dead Ending \cite{MilleR2013a, LarssMNRS2018} vs. Full Mis\`ere \cite{Siege2015b}]\label{ex:deadend}
Consider the literal form games ${\bs 1} = \cg{\boldsymbol 0}{\varnothing} $ and ${\overline {\boldsymbol 1}} = \cg{\varnothing}{\boldsymbol 0}$. In the universe of dead-ending mis\`ere-play, similar to normal-play, it holds that ${\bs 1} + {\overline {\bs 1}} =_{\E} \, {\bs 0}$, whereas in full mis\`ere-play ${\bs 1} + {\overline {\bs 1}}\; \fuzzyU{\M} \; {\bs 0}$. The Maintenance, which does not distinguish between the universes, holds. Thus, the remaining tests are in the Proviso. For dead-ending mis\`ere, it suffices to argue that ${\rm L}=o_L({\bs 1}+{\overline {\bs 1}}+X) = o_L({\bs 0}+X)$. And this holds, since $X$ is Left-atomic dead-ending: Right cannot open up any move sequence for Left in the $X$ component, when Left plays to ${\overline {\bs 1}}+X$. Moreover, ${\rm L} =o_L({\bs 0}+X)$, since X is Left-atomic. The arguments for $o_R$ are symmetric, since now $X$ is Right-atomic dead-ending.

On the other hand, for full mis\`ere, we may choose the Left-atomic distinguishing game $X=\cg{\varnothing}{\cg{{\bs 2}}{\varnothing }}$, which indeed opens up 2 new horrible moves for Left, to obtain ${\rm R}=o_L({\bs 1}+{\overline {\bs 1}}+X) < o_L({\bs 0}+X)={\rm L}$. Analogously, we may get ${\rm L} = o_R({\bs 1}+{\overline {\bs 1}}+X) > o_R({\bs 0}+X)={\rm R}$, by picking the Right-atomic distinguishing game $X = \cg{\cg{\varnothing}{{\overline {\bs 2}}}}{\varnothing }$. Hence, the confusion in terms of full mis\`ere. (There is a local justification of this by using Theorem~\ref{thm:allcomps}.)
\end{example}

Thus, we have seen the Proviso in action, and the meaning of the Left- and Right-atomic games in the Proviso should have become clearer. But, again, what is a reasonable explanation of the interaction between the Proviso and the Maintenance? This should be addressed before we plunge into the somewhat extensive and technical definitions and proofs.

The intuition is that all absolute game conventions behave in a similar fashion, with respect to maintenance of an inequality, playing in the individual games $G$ and $H$ alone as expressed in the Maintenance, except possibly at the endgame. This is easier to see whenever $H={\bs 0}$, because then Maintenance reduces to, for all $G^R$, there is a $G^{RL}$ such that $G^{RL}\su {\bs 0}$. In this way Left may ignore the $X$ component, as long as Right plays in $G$. And analysis of a move in the $X$ component may be induced by induction, unless $X$ is atomic. Whenever $H\ne \boldsymbol 0$, by lifting the framework from normal-play, then a Right move in the $H$ component may be interpreted as a conjugate of $H^L$, even if $H$ is not invertible. The Proviso, on the other hand, interprets the individual universes, and tests if the inequality still holds for atomic $X$ from the specified universe's point of view.

\subsection{Guaranteed scoring-play}\label{sec:guaranteed}
The main ideas presented so far are universal, in the sense that they do not depend on the aforementioned winning conventions.  This work sprung out of studies in {\em scoring combinatorial games}. The major novelty of a scoring universe is that for all games, each empty set of options has an adorned terminal score. Stewart \cite{Stewa2011, Stewa2019} studied the full universe of scoring combinatorial games and noticed that a special type of games are problematic, namely the so-called `hot-atomic' games. 

Suppose that Left cannot move in a game component, say $G$, i.e. $G$ is Left-atomic, and that she prefers to make no move there again. Suppose that Right is able to wait while Left moves somewhere else, and then he finds a move to a $G^R$ in a way that opens up new move options for Left inside $G^R$, and that those new options worsen her terminal score. This type of situation is analogous to our example that distinguishes the values in full mis\`ere to those of the dead-ending restriction. But for scoring games the analogue to dead-ending is the restriction called {\em guaranteed scoring-play} \cite{LarssNNS2016}, denoted $\mathbb{Gs}$, with a similar effect with respect to the waiting problem. This is achieved, not by restricting the available moves, in case of atomic games, but instead by restricting the possible terminal scores. The {\em guaranteed property} disallows hot-atomic games, at each stage of play, i.e. if a player cannot move in a game component, then the score cannot get worse for this player in this component, even if play would continue there at some later point, by the other player opening up new play possibilities.

  First a note on change of terminology: in purely move oriented games, such as normal- and mis\`ere-play, the game ${\bs 1}$ denotes one move for Left (as in Example~\ref{ex:deadend}). In a scoring universe, the number of available moves is still important, but the central concept is the score. Therefore it makes sense to instead let the game ${\bs 1}$ denote the purely atomic game, with no options for either player, and terminal score $1$, independently of who is to move. Similarly, $\bs{2}$ is the terminal game where Left is rewarded 2 points independently of who is to play, and so on.

  Let us look at the interplay of the Proviso with the Maintenance in guaranteed scoring-play. The first issue is how to interpret the Proviso in this universe. By the guaranteed property, the worst thing that can happen for Left, with respect to a Left-atomic game $X$ is that Right can control the moves in $X$ freely, and that the terminal scores are no better than the terminal score of $X$ alone, when Left starts. By the guaranteed property, then the terminal score in $X$ is the same, independently of the sequence of play, and without loss of generality, we may assume it be 0. Moreover, the issue of Right having maximal control of the moves in $X$ translates to, again by the guaranteed property, that Right can pass whenever he wishes, and that Left is forced to move only in say $G$. Therefore \cite{LarssNS2018c} the Proviso simplifies to $\underline{o}_L(G)\ge \underline{o}_L(H)$ and $\overline{o}_R(G)\ge \overline{o}_R(H)$, where the under-line (over-line) denotes that Right (Left) can pass, as an extra option at any stage of play, even as a final move. Thus such modified Right- and Left-outcomes, may be thought of as {\em waiting-protected} results in optimal play, with respect to a given starting player. In Section~\ref{sec:subord}, we will develop this idea for the mis\`ere dead-ending universe as well.

\begin{example}[Guaranteed Scoring-play]\label{ex:guaranteed}
In the guaranteed scoring universe, we have $\cg{\bs{1}}{\cg{\bs{1}}{\bs{0}}}\succcurlyeq \bs{1}$. This is justified by noting that the Maintenance is trivially satisfied if Left starts, and  
if Right starts, then $\cg{\bs{1}}{\bs{0}}$ is answered by Left moving to $\bs{1}$ which satisfies the
Maintenance. The Proviso is satisfied since
$\underline{o}_L(\cg{\bs{1}}{\cg{\bs{1}}{\bs{0}}}) = \underline{o}_L(\bs{1})  = 1$, and $ \overline{o}_R(\cg{\bs{1}}{\cg{\bs{1}}{\bs{0}}})=  \overline{o}_R(\bs{1})= 1$; in particular, Right cannot access the game ${\bs 0}$, since Left would not use the passing advantage. In contrast,
$\cg{\bs{1}}{\bs{2}}\not\succcurlyeq \bs{1}$ since the Maintenance is not satisfied; Left has no response to Right's move to
$\bs{2}$. This is analogous to $*\not\su {\bs 0}$ in normal-play. Also,
  $\cg{\cg{\bs{0}}{\bs{2} } } {\cg{\bs{1}}{\bs{0}}} \not\succcurlyeq \bs{1}$, because the Proviso does not hold, since Right's waiting-protection gives
$$\underline{o}_L(\cg{\cg{\bs{0}}{\bs{2}}}{\cg{\bs{1}}{\bs{0}}}) = 0  < 1 = \underline{o}_L(\bs{1}).$$ Observe that, in this case, the Maintenance holds. But it is irrelevant. Guaranteed scoring-play includes elements of normal-play, but is much richer.
\end{example}

By these examples, we now hope that the reader is sufficiently motivated, and prepared to dive into the notation and terminology, necessary to unify  move-oriented classical and scoring-play conventions, and exploit their common features. 

\section{Combinatorial game spaces and absolute universes}\label{sec:def}
This section has six subsections, so let us begin by outlining the content. 
Since all games are zero-sum, it is convenient to think of each terminal situation by using an associated `score', that may depend on whose turn it is to move. The {\em space} of all such games is defined in Section~\ref{sec:gameform}; followed by addition and conjugates of games. Sections~\ref{sec:gameform} and ~\ref{sec:univ} concern exclusively the {\em form} of a game, which concerns the `how-to play'. In the latter section we axiomatize the concept of a universe of games, together with parentality.  
 The evaluation of a terminal score is developed in Section~\ref{sec:eval}, together with the notion of a Combinatorial Game Space, followed by the definition of the partial outcome functions in Section~\ref{sec:outcome}, which thus addresses the `why-to play'. In Section~\ref{sec:abs}, we define outcome saturation together with the notion of an absolute universe of games, and prove that absolute is a consequence of parentality, i.e. `only form matters'. and in Section~\ref{sec:partord}, we conclude with the definition of superordinate order of games, and prove that the order is compatible with the disjunctive sum operator.
 
\subsection{Additive game forms}\label{sec:gameform}

Let $\A=(\A,+)$ be a totally ordered, additive abelian group. A terminal position will be of the form $\pura{\ell}{r}$ where $\ell, r\in \A$. The intuition, adapted from scoring game theory \cite{LarssNNS2016}, is that, if Left is to move, then the game is finished, and the `score' is $\ell$, and similarly for Right, where the `score' would be $r$. In general, if $G$ is a game with no Left options then we write $\GL=\emptyset^{\ell}$ for some $\ell\in \A$ and if Right has no options then we write $\GR=\emptyset^r$ for some $r\in \A$.\footnote{Here, the empty set is \emph{not} a game, and therefore the notation should not be confused with standard set theory, where the empty set can be an element in a set. The empty game in our notation is $\pura{\ell}{r}$ and never $\emptyset$. In our approach, the empty game carries information, and this information has to be interpreted before we are able to compare games in a given class.}

We refer to $\emptyset^a$ as an \textit{atom} and $a\in \A$ as the \textit{adorn}. A position for which at least one of the players does not have a move is called {\em atomic}, and if Left (Right)
 does not have a move is called \textit{Left- (Right-) atomic}. A \textit{purely atomic} position
 is both Left- and Right-atomic.
Let us identify  $\bs{a}=\cg{\emptyset^a}{\emptyset^a}$ for any $a\in \A$.
For example, $\bs{0}=\cg{\emptyset^0}{\emptyset^0}$ where $0$ is the neutral element of $\A$.

We separate the form of a game, which defines how to play it without incentive, from how to play it well; a  space of game forms never concerns any evaluation of results, such as a win-loss situation, which is rather the content of Sections~\ref{sec:eval} and \ref{sec:outcome}.

 \begin{definition}[Free Space of Game Forms]\label{def:freespace}
 Let $\A$ be a totally ordered abelian group and let $\Omega_0=\{\pura{\ell}{r}\mid \ell, r \in \A\}$.
 For $n>0$, $\Omega_n$ is the set of all game forms with finite sets of options in $\Omega_{n-1}$,
 including game forms that are Left- and/or Right-atomic, and the set of game forms of \textit{birthday} $n$ is $\Omega_n\setminus \Omega_{n-1}$.\footnote{The notion of birthday here is on the literal form of a game. In the literature, sometimes `birthday' concerns the game tree after reduction. Siegel \cite{Siege2013} uses the term formal birthday for the literal form birthday.}
 Let $\Omega = \bigcup_{n\ge 0} \Omega_n$. Then $\Omega = (\Omega, \A)$ is a \textit{free space} of game forms.
 \end{definition}

 Recreational combinatorial games often decompose into independent sub-positions as play progresses, e.g. {\sc domineering, clobber, konane, amazons}.
A player plays in exactly one of these sub-positions. This is one of the motivations for the {\em disjunctive sum} operator. The other motivation is that we may add any couple of games within the same convention.
We give the explicit definition in terms of the adorns. Here, and elsewhere, an expression of the type $\GL + H$ denotes the set of games  of the form $G^L+H$, $G^L\in \GL$, and this notion is only defined if $\GL$ is non-atomic.

\begin{definition}\label{def:disjsum}
Consider a free space $(\Omega,\A)$, and a pair of games $G, H\in (\Omega,\A)$. The disjunctive sum of $G$ and $H$ is given by:
\[G+H=
\begin{cases}
 \pura{\ell_1+\ell_2}{r_1+r_2} \textrm{ if } G=\pura{\ell_1}{r_1} \textrm{ and }
H=\pura{\ell_2}{r_2};\\
\cg{\emptyset^{\ell_1+\ell_2}}{\GR +H,G+\HR}, \textrm{ if }
G=\cg{\emptyset^{\ell_1}}{\GR},
H=\cg{\emptyset^{\ell_2}}{\HR},\\
 \textrm{ and at least one of $\GR $ and $\HR$
 non-atomic;}\\
\cg{\GL +H,G+\HL}{\emptyset^{r_1+r_2}}, \textrm{ if }
G=\cg{\GL}{\emptyset^{r_1}},
H=\cg{\HL}{\emptyset^{r_2}},\\
 \textrm{ and at least one of $\GL $ and $\HL$
 non-atomic;}\\
\cg{\GL +H,G+\HL}{\GR +H,G+\HR},
\textrm{ otherwise.}
\end{cases}\]
Let $((\Omega, \A),+)$ denote the induced free semigroup.
\end{definition}

Note that if $G,H\in (\Omega,\A)$, a free space, then, by Definition~\ref{def:disjsum}, $G+H\in (\Omega,\A)$.\footnote{ We use the symbol `+' for both addition in $(\A, +)$ and
for the disjunctive sum $(\Omega , +)$, defined by the surrounding context; then, a correct, but too heavy, notation would be $((\Omega, (\A , +)),+)$, where the brackets determine the respective meanings of `+'.} Since our structure is associative, it is a commutative semigroup. The proof of these facts is analogous to that given for Guaranteed scoring games \cite[Theorem 7, page 6]{LarssNNS2016}; commutativity is easy to see by Definition~\ref{def:disjsum}, since $\A$ is abelian.

The game $\bs{0}=\cgs{\emptyset^0}{\emptyset^0}\in (\Omega,\A)$
for any $\A$. In any circumstance, it should be the case that $G+\bs{0}$ and $G$ are identical. Since $\bs{0}$ represents the empty game, there is no move, and the `score' is 0 regardless of whose turn it is. We prove it to give an example of `induction on options', and also to establish the neutral element of our structure so that the semigroup is in fact a monoid. We will write $G\equiv H$ if $G$ and $H$ are identical, that is their literal forms have the same followers.

\begin{lemma}\label{lem:id}
Let $G \in ((\Omega,\A), +)$. Then $G+\bs{0}\equiv G$.
\end{lemma}
\begin{proof}If $G$ is purely atomic, then $G=\pura{\ell}{r}$, for some $\ell,r\in \A$. Then
 $G+\bs{0} = \pura{\ell}{r}+\pura{0}{0}=\pura{\ell}{r}$, by Definition~\ref{def:disjsum}.

 Let $G=\cgs{\GL}{\GR}$ be of birthday at least one.
If $G=\cgs{\emptyset^\ell}{\GR}$ then  $G+\bs{0} = \cgs{\emptyset^{\ell}}{\GR}+\pura{0}{0} = \cgs{\emptyset^{\ell}}{\GR+\bs{0}}$,
if $G=\cgs{\GL}{\atom{r}}$ then  $G+\bs{0} = \cgs{\GL}{\atom{r}} +\pura{0}{0} = \cgs{\GL+\bs 0}{\atom{r}}$,
and if $G=\cgs{\GL}{\GR}$ then $G+\bs{0} = \cgs{\GL}{\GR}+\pura{0}{0} = \cgs{\GL+\bs{0}}{\GR+\bs{0}}$, all by the definition of disjunctive sum.

By induction, we have that  $\GL+\bs{0}\equiv\GL$ and also $\GR+\bs{0}\equiv\GR$.
Therefore $G\equiv G+\bs{0}$.
\end{proof}
In Section \ref{sec:partord} we define a partial order on this monoid.

The \emph{conjugate} of a game denotes another game where the only difference is that Left and Right have `switched roles'.
\begin{definition}\label{def:conjugate}
The conjugate of $G\in \Omega$ is
\[ \conjt{G} \, =
\begin{cases}
\pura{-r}{-\ell}, &\mbox{if $G=\pura{\ell}{r}$, $\ell,r\in \A$}\\
\cgs{\conj\GR\,}{\emptyset^{-\ell} }, &\mbox{if $G=\cgs{\emptyset^{\ell}}{\GR}$}\\
\cgs{\emptyset^{-r}}{\;\conj\GL }, &\mbox{if $G=\cgs{\GL}{\emptyset^{r}}$}\\
\cgs{\conj\GR \,}{\;\conj\GL }, &\mbox{otherwise},
\end{cases}
\]
where $\conj{\GL}$ denotes the set of games
$\conj{G^L} $, for $G^L\in \GL$, and similarly for $\GR$.

\end{definition}
By the recursive definition of the free space $(\Omega,\A)$, each combinatorial game space is closed under conjugation.
In normal-play, the games form an ordered group and each game $G$ has an additive inverse,
appropriately called $-G$ and $-G = \,\conjt{G}$. However, there are other spaces of games, for example scoring and mis\`ere
games, where $\conjt{G}$ is not necessarily $-G$ (see \cite{Mille2015}).

\subsection{Universes of games and parentality}\label{sec:univ}
Many results in the literature concern strictly smaller structures than a free space. In dicot scoring-play the classical examples are Milnor's and Hanner's non-negative incentive games \cite{Milno1953}, and later Ettinger studied all dicot scoring games \cite{Ettin1996}. Johnson \cite{Johns2011} studies dicot scoring games with length of play of fixed parity. The normal-play theory started with impartial rulesets that do not distinguish between the players \cite{Bouto1902}, and another classical normal-play restriction concerns all-small (dicot) games \cite{BerleCG2001}. Dorbec et al. study Dicot Mis\`ere-play \cite{DorbeRSS2015}, whereas Milley and Renault study dead-ending mis\`ere-play  \cite{MilleR2013a}. Plambeck's sets of impartial games generated via disjunctive sums on single rulesets \cite{Plamb2009} are universes, although not absolute. (As we will see in Section~\ref{sec:abs}, all except \cite{Bouto1902}, \cite{Milno1953}, \cite{Johns2011} and \cite{Plamb2009} fit our theory.)

\begin{definition}[Universe]\label{def:universe}
A universe of games, $((\U,\A),+)$ satisfies $(\U,\A)\subseteq (\Omega,\A)$, for some free monoid $((\Omega,\A), +)$, with
\begin{enumerate}
\item $\pura{a}{a}\in \U$ for all $a\in \A$;
\item {\em options closure}: if $G\in \U$ and $H$ is an option of $G$ then $H\in \U$;
\item {\em disjunctive sum closure}: if $G,H\in \U$ then $G+H\in \U$;
\item {\em conjugate closure}: if $G\in \U$ then $\conjt{G}\,\in \U$;
\end{enumerate}
\end{definition}
Note that the neutral element, associativity and commutativity are carried over from the free space $(\Omega, \A)$, and so each universe $\U=((\U,\A) +)$ is also a commutative monoid.

Studied sets of games are often closed under taking options (a.k.a. hereditary closure), disjunctive sum (additive closure), and conjugation (flip the game board). The next property is somewhat less common.
\begin{definition}[Parentality]\label{def:parental}
A universe $\U$ of combinatorial games is parental if, for any pair of finite non-empty sets of games,
 $\cal G, \cal H\subset \U$, then
 $\cg{\cal G}{\cal H}\in \U$.
\end{definition}

Not all universes are parental. For example, no universe of impartial games (players have the same options) is parental. No universe of the form of Milnor's positional games is parental. The restriction that it is never bad to start is disqualified by the generality of the sets $\mathcal G$ and $\mathcal H$.  All dicot universes are parental and any dead-ending universe is parental.  The restriction to dead-ending games concerns games with empty sets of options, and hence the parental property is not affected. A similar argument holds for the monoid of guaranteed scoring games. Each free space is parental (by definition).

\subsection{Evaluation maps relating to standard game spaces}\label{sec:eval}
So far, only the `form' of a game has been considered. In this section, we elaborate on the incentive to play.

\begin{definition}[Evaluation Map]\label{def:eval}
The result of a terminal game is evaluated  via one of two order preserving maps $\nu_L,\nu_R:\A\rightarrow \s$, depending on who is to move; if Left (Right) is to move in a Left-atomic (Right-atomic) game, with adorn $\ell\in \A$ ($r\in \A$), then the result is $\nu_L(\ell)$ ($\nu_R(r)$).
\end{definition}
 In play Left (Right) seeks to maximize (minimize) the result. Indeed, these evaluation maps will be extended to optimal play outcome functions in Definition~\ref{def:outcome}.\footnote{Observe that we could not have directly set the adorns to +1 and -1 respectively, to model for example that Left wins or loses normal-play. Namely this cannot model disjunctive sum play, where there is no concept of win/loss on individual game components. But luckily, this idea is instead neatly captured by the $\nu$ function. In scoring-play however, `the winner' is not the relevant concept, and again $\nu$ is flexible enough to instead capture the sum of the terminal component scores. See Lemma~\ref{lem:outcome1} for more on this topic.}

Let us first gather what we have got so far, by combining a Free Space with an evaluation function, in the notion of  a Combinatorial Game Space.

\begin{definition}[Combinatorial Game Space]\label{def:space}
A combinatorial game space, abbreviated $\Omega$, is the free commutative monoid $((\Omega, \A), +)$, together with a totally ordered set $\s$ and two order preserving evaluation maps $\nu_L:\A\rightarrow \s$ and $\nu_R:\A\rightarrow \s$. That is
$$\Omega=((\Omega,\A),\s, \nu_L, \nu_R,+),$$
where `+' is the disjunctive sum over the free space $(\Omega, \A)$. If $| \A | > 1$ then $(\Omega, \A)$ is a scoring-play space, with $\nu(a) = \nu_L(a) = \nu_R(a)$, for all $a\in \A$.
\end{definition}

\begin{observation}\label{obs:lit}
In the literature, combinatorial games are normal-play, mis\`ere-play or scoring-play.

\begin{enumerate}
\item \textit{Normal-play} corresponds to the trivial group $\A=\{0\}$ and the set $\s=\{-1,+1\}$, together with the maps
$\nu_L(0)=-1$, $\nu_R(0) =+1$,
\item \textit{Mis\`{e}re-play} corresponds to the trivial group $\A=\{0\}$ and the set $\s=\{-1,+1\}$, together with the maps  $\nu_L(0) = +1$, $\nu_R(0) =-1$,
\item \emph{Scoring-play} corresponds to $\A=\s=\mathbb R$,
with its natural order and addition, and with $\nu$, the identity map.
\end{enumerate}
\end{observation}

A universe $\U\subseteq (\Omega,\A)$ inherits the totally ordered set $\s$ and the two order preserving evaluation maps $\nu_L:\A\rightarrow \s$ and $\nu_R:\A\rightarrow \s$ from its enclosing free space, which defines the structure $\U=((\U,\A),\s, \nu_L, \nu_R,+)$.

Only `forms' are required to define universes of games, but one mostly thinks of a universe of games in terms of an associated winning convention. These topics are expanded upon in the paper \cite{LarssNSabsinf}.

\begin{remark}It is easy to imagine scoring-play games that do not satisfy our axioms. For example, the players score points in some board game, but the result depends on the final score difference modulo 3.
We might have $\A=\{0,1,2\}$, where addition is modulo 3; $\s=\{-1,0,+1\}$, $\nu(0) = 0$, $\nu(1) = +1$,
and $\nu(2) = -1$, and so $\nu$ is not order preserving.
\end{remark}

\subsection{The outcome function}\label{sec:outcome}
The mappings of adorns in $\A$ to elements of $\s$, vial the evaluation maps $\nu_L$ and $\nu_R$ from Definition~\ref{def:space} is extended to generic positions via two recursively defined \emph{partial outcome functions}. 
The definition is general enough to match the usual definitions of outcomes for several universes studied in literature, in the sense of providing a generalization of Observation~\ref{obs:lit} to any game $G$. The outcome functions are not sensitive to the specific universe, only to the game space.

\begin{definition}\label{def:outcome}
Let $\Omega$ be a combinatorial game space, and consider given evaluation maps $\nu_L:\A\rightarrow \s$ and $\nu_R:\A\rightarrow \s$. The Left- and Right-outcome functions are $o_L:\Omega \rightarrow \s, o_R:\Omega\rightarrow \s$,
where
$$o_L(G) = \begin{cases}
\nu_L(\ell) & \textrm{if  $G = \cg{\emptyset^\ell}{\GR}$,} \\
 \max_L\{o_R(G^L)\} & \textrm{otherwise}
\end{cases}
$$

$$o_R(G) = \begin{cases}
\nu_R(r) & \textrm{if  $G = \cg{\GL}{\emptyset^r}$,} \\
 \min_R\{o_L(G^R)\} & \textrm{otherwise,}
\end{cases}
$$
where the $\max_L$ ($\min_R$) ranges over all Left (Right) options of the given game $G\in \Omega$. The outcome of $G$ is the ordered pair of Left- and Right-outcomes $o(G) = (o_L(G), o_R(G))$.
\end{definition}

A brief note on terminology: the function $o_L$ is here called the Left-outcome, and this function concerns optimal play results, when player Left starts. Usually, when we refer to `outcome' in CGT, the move-flag has not yet been assigned; hence this function is necessarily an  ordered pair $o = (o_L, o_R)$. The outcome function turns a universe of games into a partially ordered monoid, and this we show in Section~\ref{sec:partord}.

Let $G\in \U$. From Definition \ref{def:outcome} we have that $o_L(G)=\nu_i(\ell)$ and
$o_R(G)=\nu_j(r)$ for some $\ell,r\in\A$, and $i,j\in\{L,R\}$. Therefore we may always assume that the set of results is $\s=\{ \nu_L(a):a\in\A\}\cup\{\nu_R(a):a\in \A\}$. If the union of the images were smaller than the union of the codomains, then we restrict $\s$ to be the union of the images. With this convention, we get the following trivial observation.

\begin{lemma}\label{lem:outcome0} If $|\A| = 1$ then $|\s| \le 2$.
\end{lemma}
\begin{proof} If $|\A|=1$ then $\s$ contains exactly $\nu_L(0)$ and $\nu_R(0)$, at most 2 values.
\end{proof}
Note that, if $|\s|=1$, then play is trivial.

If one component in a disjunctive sum is a purely atomic game, then the outcome functions are easy to calculate.
We state them for Left-outcomes only.

\begin{lemma}\label{lem:outcome1} Let $G, H \in \U$, a universe of combinatorial games, and let $\bs{c}~=~\cgs{\atom{c}}{\atom{c}}$.

\begin{itemize}
\item If $G = \cgs{\atom{a}}{\GR}$, $H = \cgs{\atom{b}}{\HR}$, then $o_L(G+H) =\nu_L(a+b)$;
\item If $o_L(G)=\nu_L(a)$, then $o_L(G+ \bs{c}) = \nu_L(a+c)$;
\item $o_L(G+ \bs{c})\ge o_L(H+ \bs{c})$ if and only if $o_L(G)\ge o_L(H)$.
\end{itemize}
\end{lemma}

\begin{proof}
These follow directly from Definition~\ref{def:disjsum} and Definition~\ref{def:outcome}.
\end{proof}

\subsection{Saturation, parentality and absolute universes}\label{sec:abs}
The {\em adjoint} of a game is similar to the conjugate. We will have use for it in Theorem~\ref{thm:pardense}, where we prove that parentality implies saturation in absolute universes. The following definition is a generalization of the standard concept from mis\`ere play theories.

\begin{definition}[Adjoint]\label{def:adjoint}
Consider a universe $\U$ and a game $G\in \U$. Then the adjoint of $G$ is the game
\[
G^\circ =
\begin{cases}
\cg{-{\bs r}}{-{\bs \ell}} & \textrm{if  $G = \pura{\ell}{r}$}; \\
\cg{\GR^\circ}{-{\bs \ell}} & \textrm{if $|\GR|>0$ and $\GL=\atom{\ell}$}; \\
\cg{-{\bs r}}{\GL^\circ} & \textrm{if $|\GL|>0$ and $\GR=\atom{r}$}; \\
\cg{\GR^\circ}{ \GL^\circ} & \textrm{otherwise},
\end{cases}
\]
and where $\circ$ applied to a set operates on its elements.
\end{definition}

In all universes studied in the literature, for any game $G$, it is possible to conceive a disjunctive sum $G+H$ where the first player has the advantage. Let $*=\cg{{\bs 0}}{{\bs 0}}$.  For example, regarding normal-play, take $H=-G+*$; regarding mis\`ere-play, take $H=\{G^\circ |G^\circ\}$; regarding scoring-play, take $H=\cg{\bs s}{-\bs s}$, with a sufficiently large $s$, or in terms of the adjoint let $H=\{G^\circ+\bs 1 |G^\circ-\bs 1\}$. The following definition formalizes this idea in terms of {\em saturation}.

\begin{definition}[Outcome Saturation]\label{def:dense}
A universe $\U$ of combinatorial games is saturated if, for all $G\in \U$, for any pair $x,y\in \s$,
there is an $H\in \U$ such that $o_L(G+H)\ge x$ and $o_R(G+H)\le y$.
\end{definition}

\begin{definition}[Absolute Universe]\label{def:cgs}
A universe is \abs if it is parental and saturated.
\end{definition}

We will prove that absolute universes have rich algebraic structures, and the naming is inspired by \emph{neutrality} with respect to `winning condition'.\footnote{Naming is inspired by ``neutral'' or ``absolute geometry'', a relaxation of Euclidean geometry, that provides an umbrella for many geometries. By analogy, in absolute geometry, the four first postulates together give a rich theory, and it is neutral with respect to the parallel postulate.}

It turns out that saturation is a consequence of parentality. 

\begin{theorem}\label{thm:pardense}
Parental universes are outcome saturated.
\end{theorem}
\begin{proof}
Consider a parental universe $\U\subset \Omega$. First of all note that if $G\in \U$, then $\conjt{G},G^\circ\in\U$.

Consider the case of $\U$ scoring-play, and let $G\in \U$. Then the evaluation map is independent of player. (The result may of course depend of who is to play, as is the case for $o(\pura{\ell}{r})=(\nu(\ell),\nu(r))$.)
Consider given $x,y\in S$. Since $\nu$ is order preserving, there is an $\ell \in \A$ such that $\nu(\ell)\ge x$, and there is an $r\in \A$ such that $\nu(r)\le y$. 
Let $H = \,\conjt{G}+\cg{\bs \ell}{\bs r}$, and note that, since $\U$ is parental, $\cg{\bs \ell}{\bs r}\in \U$.  If Left starts, she can play to $G+\!\conj{G}+\bs \ell$, and mimic all remaining moves. This proves that $o_L(G+H)\ge \nu(\ell)\ge x$. The proof of $o_R(G+H)\le y$ is similar. 

Next, we consider the case when $\A=\{0\}$ and $\s=\{-1,1\}$. It suffices to prove that, for any $G\in \U$, there is an $H\in \U$ such that $o_L(G+H)=1$ and $o_R(G+H)=-1$. Suppose first that $\nu_L(0)=1$ and $\nu_R(0)=-1$, i.e. the mis\`ere-play convention. Take $H=G^\circ+*$, and note that $H\in \U$ since $\U$ is parental. Left can play to $G+\,G^\circ$, and mimic Right until he makes the last move. This shows that $o_L(G+H)=\nu_L(0)=1$. Similarly, $H=G^\circ+*$ gives $o_R(G+H)=\nu_R(0)=-1$.

Next, suppose that $\nu_L(0) = -1$ and $\nu_R(0) = 1$ (i.e. normal-play). Then instead take $H=\;\conjt{G}+\,*\in \U$, by parentality. Left plays to $G\,+\!\conjt{G}$ and mimics Right, until she takes the last move. This shows that $o_L(G+H)=\nu_R(0)=1$, and similarly $o_R(G+H)=\nu_L(0)=-1$.
\end{proof}

Note that for normal- and mis\`ere-play the proof gives that there is a game $H$ that satisfies `Next player wins'. Similarly one can find $H$ satisfying the other outcome classes.
\begin{corollary}
Any parental normal- or mis\`ere-play universe $\U$ satisfies for all $G\in \U$, for any $x,y\in\{-1,1\}$, there is an $H\in \U$ such that $o_L(G+H)=x$ and $o_R(G+H)=y$.
\end{corollary}
 \begin{proof}
 Modify the game $H$ in the second part of the proof of Theorem~\ref{thm:pardense}. We give more details in Observation~\ref{obs:misout}.
 \end{proof}
Hence, for normal- and mis\`ere-play, saturation has this stronger notion.
For scoring-play, one may obtain a stronger notion in the following spirit.

\begin{corollary}
A scoring-play universe satisfies for all $G\in \U$, for any given pair $x,y\in\s$, there is an $H\in \U$ such that $o_L(G+H)\ge x$ and $o_R(G+H)\le y$, and there is a $K\in \U$ such that $o_L(G+K)\le x$ and $o_R(G+K)\ge y$.
\end{corollary}
 \begin{proof}
 This is a consequence of the first part of the proof of Theorem~\ref{thm:pardense}, but for the second part of the statement, use instead $H={G}^\circ+*$.
 \end{proof}
The important corollary is the next one. It proves that only form matters, when it comes to the validation of whether a universe is absolute.  
\begin{corollary}\label{cor:parabs}
Any parental universe is absolute.
\end{corollary}
\begin{proof}
This is a direct consequence of Theorem~\ref{thm:pardense}.
\end{proof}

\subsection{Order in universe}\label{sec:partord}
In general, two games $G$ and $H$, are {\em equal} if both players are indifferent to playing $G$ or $H$ in any `situation', which  means in any disjunctive sum since we are discussing combinatorial games. This idea is usually extended to define a partial order on games.

\begin{definition}\label{def:order}(Superordinate Game Comparison) 
Let $\U\subset\Omega$ be a universe of combinatorial games. For $G, H \in \U$, $G\su H $ modulo $\U$, or $G\su_{\U} H $, if and only if
$o_L(G+ X)\ge o_L(H+ X)$ and $o_R(G+ X)\ge o_R(H+ X)$, for all games $X\in \U $. And $G=H$ if $G\su H $ and $H\su G $.
\end{definition}

Definition \ref{def:order} can be extended to have the distinguishing games in a sub- or super-set of $\U$.
 Suppose, for example, that $G=H$ are Milnor type games.
Then $G, H$ could perhaps be distinguished in the general dicot universe, the guaranteed universe, or perhaps
in the free space of all scoring combinatorial games, and, in general, the relation between the games is not the same. The importance of `modulo $\U$' will be expanded upon in Section~\ref{sec:subordinate}.
When the context is clear, however, instead of the phrase ``for $G, H \in \U$, $G\su H $ modulo $\U$", we often write simply ``for $G, H \in \U$, $G\su H$".

The next result shows that the order is compatible with the disjunctive sum operator, so that a
universe is in fact a partially ordered quotient monoid.

\begin{theorem} \label{sumorder} Let $G,H\in \U$.
If $G\su H$ then, for all $J\in\U$, $G+J\succcurlyeq H+J$.
\end{theorem}
\begin{proof}Consider any game $J\in \U$. Since $G\su H$, it follows that, $o_L(G+(J+X))\ge o_L(H+(J+X))$,
for any $X\in \U$.
Since disjunctive sum is associative this inequality is the same as
$o_L((G+J)+X))\ge o_L((H+J)+X)$. The same argument gives
$o_R((G+J)+X)\ge o_R((H+J)+X)$ and thus, since $X$ is arbitrary, this gives that $G+J\su H+J$.
\end{proof}

It is possible to have quasi-identities
such that $X\neq 0$ and $X+X=X$. Consider the universe of impartial games with
the mis\`{e}re-play convention. It is well known that, for $X=*2+*2\neq 0$, we have $X+X=X$. The full equivalence is related to invertibility as follows; the proof is similar to that of Theorem \ref{sumorder}.

\begin{corollary} \label{sumorder2} Let $G,H\in \U$. Then, for all invertible $J\in\U$,
$$G\succcurlyeq H\Leftrightarrow G+J\succcurlyeq H+J.$$
\end{corollary}

\section{The main lemma}\label{sec:downlinked}
This section concerns lemmas for Theorems \ref{thm:cnp} and \ref{thm:basic}.
They compensate a potential loss of group structure in an absolute universe.

\begin{lemma} \label{lem:and} Let $\U$ be an \abs universe, and suppose that $G,H\in \U$. If $G\not\succcurlyeq H$ then
\begin{enumerate}[]
  \item $o_L(G+T)<o_L(H+T)$, for some $T\in\U$, and
  \item $o_R(G+V)<o_R(H+V)$, for some $V\in\U$.
\end{enumerate}
\end{lemma}

\begin{proof}
Because $G\not\succcurlyeq H$, at least one of the two inequalities must hold. Without loss of generality, we may
assume the first.
Thus, let $T$ be such that $$o_1=o_L(G+T)<o_L(H+T)=o_2,$$
and we have to find a $V\in \U$ such that
\begin{align}\label{eq:toprove}
o_R(G+V)<o_R(H+V).
\end{align}
Since $\U$ is outcome saturated, for every $H^{R}\in H^{\mathcal{R}}$ there exists a game ${H^R}'=(H^R)' \in\U$ such that

\begin{align}\label{eq:o2}
o_2\le o_R({H^R}'+{H^R}).
\end{align}

Let ${\HR}'= \{{H^{R}}' : H^R\in H^{\mathcal{R}}\}$, possibly the empty set if $H$ is Right-atomic.

Since $\bs{0}\in \U$ (by definition of a universe) and since $\U$ is parental then the game
$$V=\cgs{{\HR}' ,\bs{0}}{T}\in \U.\footnote{If $\HR$ is non-empty, then the Left option $\bs{0}$ can be omitted.}$$
We claim that
\begin{align}\label{eq:GV}
o_R(G+V)\le o_1,
\end{align}
 and
\begin{align}\label{eq:HV}
o_2\le o_R(H+V),
\end{align}
regardless of $\HR$.
In the game $G+V$, Right has a move in $V$ to $G+T$, and so inequality \eqref{eq:GV} holds.

If Right's best move in $H+V$ is to $H+T$, then $o_L(H+T)=o_2$ implies \eqref{eq:HV}. If Right's best move in $H+V$ is to $H^R+V$, then Left can respond to $H^R+{H^R}'$, and so, by \eqref{eq:o2}, $o_2\le o_R(H^R+{H^R}')\le o_L(H^R+V)= o_R(H+V)$. Thus inequality \eqref{eq:HV} holds.
\end{proof}

In normal-play universes it is automatically true that $G\,+\conjt{G}\; =\bs{0}$ and that $G\succcurlyeq H$ is equivalent to
$G\,+ \conj{H}\, \succcurlyeq \bs{0}$. To get around the problem of non-invertible elements in the other absolute universes, we use the following concept, introduced independently in \cite{Ettin1996} for scoring games, and in \cite{Siege2015b} for mis\`ere games. It is defined for any universe.

\begin{definition}\label{def:downlinked} Let $G,H\in\U\subseteq(\Omega,\A )$. Then $G$ is \emph{downlinked} to $H$, denoted $G\!\swarrow\! H$, if $o_L(G+T)<o_R(H+T)$ for some $T\in \U$.
\end{definition}
\begin{lemma}[Main Lemma]\label{lem:main} Let $G, H\in \U\subseteq(\Omega,\A )$. If $\U$ is absolute, then
\begin{center}
$G\!\swarrow\! H$ if and only if $\forall G^L$, $G^L\not\succcurlyeq H$, and $ \forall H^R$, $G\not\succcurlyeq H^R$.
\end{center}
\end{lemma}

\begin{proof}
If $|\s|\le 1$ then there is nothing to prove so we may assume that $|\s|>1$. If $|\A|=1$, that is $\A=\{0\}$, then, by Lemma \ref{lem:outcome0}, $|\s|=2$, and so without loss of generality we may assume that $\s=\{-1,1\}$. This gives one out of two cases, normal- or mis\`ere-play, as displayed in the examples of universes in Section~\ref{sec:def}, and the results are known, in particular Siegel solved the case of mis\`ere-play \cite[p274]{Siege2015b}; we include these cases for completeness, and to illustrate the core idea before plunging into the harder case for scoring-play. The forward direction ignores play convention.\\

\noindent ($\Rightarrow$) Suppose $G\!\swarrow\!  H$, that is $o_L(G+T)<o_R(H+T)$ for some $T\in \U$. Then, for any $G^L$ and any $H^R$:
 \begin{eqnarray*}
   o_R(G^L+T)&\leqslant& o_L(G+T)<o_R(H+T)\mbox{, and so }G^L\not\succcurlyeq H;\\
   o_L(G+T)&<&o_R(H+T)\leqslant o_L(H^R+T)\mbox{, and so }G\not\succcurlyeq H^R.
 \end{eqnarray*}

\noindent ($\Leftarrow$) Suppose $\forall G^L$, $G^L\not\su H$, and $ \forall H^R$, $G\not\su H^R$. Hence, by Lemma~\ref{lem:and}, for each $G^{L_i}$, we have $X_i\in\U$ such that

\begin{align}\label{eq:central}
o_L(G^{L_i}+X_i)<o_L(H+X_i)
\end{align}

and, for each $H^{R_j}$, we have $Y_j$ such that

\begin{align}\label{eq:central2}
o_R(G+Y_j)<o_R(H^{R_j}+Y_j).
\end{align}

Consider first the case of $\mathcal A=\{0\}$, and let $\mathcal X=\{X_i\}$ and $\mathcal Y=\{Y_i\}$ be the sets of all such $X_i$s and $Y_i$s.

We must prove that $G\!\swarrow\! H$, i.e. we must find $T$, such that $-1=o_L(G+T)<o_R(H+T)=1$, which may be interpreted as second player wins in either case. By symmetry it suffices to demonstrate one of these, say $-1=o_L(G+T)$,  for normal- and mis\`ere-play respectively.\label{normal}

Define
$$T_{{\rm nor}}=
  \left\{
    \begin{array}{cl}
        {\bs 0} & G=H={\bs 0};\\
                \cgs{\atom{0}}{\,\conj{\HL}} & G={\bs 0}, \HL\neq\emptyset, \HR=\atom{0};\\
                \cgs{\conj{\GR}\,}{\atom{0}} & H={\bs 0}, \GL=\atom{0}, \GR\neq\emptyset;\\
       \cg{\mathcal Y}{\mathcal X} &  \HL\neq\emptyset, \GR\neq\emptyset.\\
    \end{array}
\right.
$$

Define
$$T_{{\rm mis}}=
  \left\{
    \begin{array}{cl}
        * & G=H={\bs 0};\\
                \cg{\bs{0}}{{\HL}^\circ} & G={\bs 0}, \HL\neq\emptyset, \HR=\atom{0};\\
                      \cg{{\GR}^\circ}{\bs 0} & H={\bs 0}, \GL=\atom{0}, \GR\neq\emptyset;\\
       \cg{\mathcal Y}{\mathcal X} & \HL\neq\emptyset, \GR\neq\emptyset.\\
    \end{array}
\right.
$$
In the case of normal-play, we use $T_{{\rm nor}}$ and prove in each case that Right wins $G+T$ when Left starts. In the first two cases, Left loses immediately. In the third case, Right can mimic Left's moves until end of play and hence Left loses. For the fourth case, if Left plays to $G+Y_j$, then \eqref{eq:central2} implies that Left loses, and if Left plays to $G^{L_i}+T$, then Right can respond to $G^{L_i}+X_i$, and Left loses, by \eqref{eq:central}.

The argument for mis\`ere-play is almost identical, but use instead $T_{{\rm mis}}$.

In the case of scoring-play, we use a \emph{centralization} idea. Recall that, by Definition~\ref{def:space},
for $a\in \A$, $\nu_L(a)=\nu_R(a)$.
For all Left options of $G$, let $\bs{a}_i=\pura{a_i}{a_i}$ be such that $o_L(G^{L_i}+X_i)=\nu(a_i)$, and, for all Right options of $H$, let $\bs{b}_j=\pura{b_j}{b_j}$ be such that $o_R(G+Y_j)=\nu(b_j)$.

Now, by Lemma \ref{lem:outcome1}, for each $i$,
\begin{align}\notag
o_L(G^{L_i}+X_i)<o_L(H+X_i) \Leftrightarrow
o_L(G^{L_i}+X_i-\bs{a_i})<o_L(H+X_i-\bs{a_i}).
\end{align}

\noindent
For all $i$, let $X_i'=X_i-\bs{a_i}$, and for all $j$, let $Y_j'= Y_j-\bs{b_j}$. Let $\mathcal X'=\{X_i'\}$ and $\mathcal Y'=\{Y_i'\}$ be the sets of all such thus centralized $X_i'$s and $Y_i'$s.

We know that, for all $i$, $-\bs{a_i}$ exists, and that  $X_i'\in\U$ because
$\U$ is closed under disjunctive sum.
Then, by Lemma~\ref{lem:outcome1}, for all $i$,
\begin{align}\label{eq:important}
o_L(G^{L_i}+X_i')=\nu(0) < o_L(H+X_i'),
\end{align}
where the inequality is by definition of $X_i'$ and by (\ref{eq:central}). Analogously, by (\ref{eq:central2}), for all $j$, we get
\begin{align}\label{eq:important2}
o_R(G+Y_j') = \nu(0) < o_R(H^{R_j}+Y_j').
\end{align}

 Pick any $\bs{t}=\cg{\emptyset^t}{\emptyset^t}$ and $\bs{s}=\cg{\emptyset^s}{\emptyset^s}$ with $t<s$.
By outcome saturation, for each $i$ and $j$, we can find $G^{R_i,(t)}$ and $H^{L_j,(s)}$ such that
\begin{align}\label{eq:nut}
o_R(G^{R_i,(t)}+G^{R_i}) \le \nu(t)
\end{align}
and such that $o_L(H^{L_j,(s)}+H^{L_j}) \ge\nu(s)$. Also, let $G^{(0)}$ and $H^{(0)}$ be the games with
 $o_R(G^{(0)}+G)\le \nu(0) \le o_L(H^{(0)}+H)$.
Define
$$T=
  \left\{
    \begin{array}{cc}
        \cg{\cg{G^{\mathcal R,(t)},\bs{0}}{\bs{t}-\bs{\ell}}}{\cg{\bs{s}-\bs{r}}{H^{\mathcal L,(s)},\bs{0}}} & G^\mathcal{L}=\emptyset^\ell\,\,and\,\,H^\mathcal{R}=\emptyset^r;\\
                \cg{G^{(0)}}{\mathcal X'} & G^\mathcal{L}\neq\emptyset\,\,and\,\,H^\mathcal{R}=\atom{r};\\
                      \cg{\mathcal Y'}{H^{(0)}} & G^\mathcal{L}=\atom{\ell}\,\,and\,\,H^\mathcal{R}\neq\emptyset;\\
       \cg{\mathcal Y'}{\mathcal X'} & G^\mathcal{L}\neq\emptyset\,\,and\,\,H^\mathcal{R}\neq\emptyset.\\
    \end{array}
\right.
$$
Because of the parental property, $T\in \Omega$.

Now, we will argue that $G\!\swarrow \! H$ via $T$. There are four cases,
one for each case in the definition of $T$.

\textbf{[1]:} In the first case, we claim that
  \begin{align}\label{eq:doubleatom}
  o_L(G+T) = o_R\left(G+\cg{G^{\mathcal R,(t)},\bs{0}}{\bs{t}-\bs{\ell}}\right)\le\nu(t).
  \end{align}
  That $\nu(s)\le o_L(H+\cg{\bs{s}-\bs{r}}{H^{L_j,(s)},\bs{0}}) = o_R(H+T)$ follows by the same argument. By the
  definition of $t<s$, and since $\nu$ is order preserving, together this implies that  $o_L(G+T)<o_R(H+T)$.
  To prove the claim, the first equality in (\ref{eq:doubleatom}) follows since there is no Left move in $G$.
  For the inequality, if Right moves in $G$, then Left has a response to $G^{R_i}+G^{R_i,(t)}$,
  with outcome as in (\ref{eq:nut}). Right can play to $G+\bs{t}-\bs{\ell}$, which proves this case,
  since $o_L(G+\bs{t}-\bs{\ell})=\nu (\ell+t-\ell)=\nu(t)$.

\textbf{[2]:} In the second case,
      $$o_L(G+T)\le \nu(0)<o_L({H+X_i'}) = o_R(H+T).$$
      The first equality follows, because if Left moves to $G+G^{(0)}$, then $o_R(G+G^{(0)})\le\nu(0)$. If Left moves to $G^{L_i}+T$, then Right has a response to $G^{L_i}+X_i'$, with $o_L(G^{L_i}+X_i')  = \nu(0)$, by (\ref{eq:important}), so Left cannot do better than $\nu(0)$. The equality follows because Right does not have a move in $H$.

\textbf{[3]:}  This argument is in analogy with the second case.

\textbf{[4]:} In case that Left has an option in $G$ and Right has an option in $H$, the following diagram shows all the cases. Possible best play for Left in $G+T$ is indicated with the uppermost `or' (and similarly for Right's play in the game $H+T$ in the bottom). We use the centralization to $\nu(0)$ to bound the desired left outcomes.

 \begin{center}
 \psset{xunit=1.0cm,yunit=1.0cm,algebraic=true,dotstyle=o,dotsize=3pt 0,linewidth=0.8pt,arrowsize=3pt 2,arrowinset=0.25}
 \begin{pspicture*}(0.4,-4.8)(12,5.2)
 \rput[tl](4.26,5.1){$ o_L(G+T) $}
 \psline{->}(5.02,4.4)(1.76,2.18)
 \psline{->}(5.02,4.4)(8.5,2.28)
 \rput[tl](2.86,3.62){$=$}
 \rput[tl](7.04,3.7){$=$}
 \rput[tl](4.88,3.62){`{\text or}'}
 \rput[tl](0.98,1.9){$ o_R(G^{L_i}+T) $}
 \rput[tl](6.76,1.9){$ o_R(G+Y_j') = \nu(0)$}
 \rput[tl](1.7,1.1){$\le$}
 \rput[tl](0.5,0.33){$ o_L(G^{L_i}+X_i') = \nu(0)$}
 \rput[tl](4.66,-4.18){$ o_R(H+T) $}
 \psline{->}(5.36,-3.72)(1.7,-2.38)
 \psline{->}(5.36,-3.72)(8.84,-2.32)
 \rput[tl](2.86,-3.24){$=$}
 \rput[tl](7.22,-3.32){$=$}
 \rput[tl](5.0,-2.58){\text{ `or'}}
 \rput[tl](7.64,-1.64){$ o_L(H+X_{m}') $}
 \rput[tl](4.8,-.5){``$<$''}
 \rput[tl](0.86,-1.84){$ o_L(H^{R_k}+T)$}
 \psline[linestyle=dashed,dash=1pt 2pt](8.5,-0.9)(2.2,-1.4)
 \end{pspicture*}
 \end{center}
We claim that, by construction, the Left outcomes below the dotted line reside \emph{above} $\nu (0)$. Namely, by (\ref{eq:important2}), a (perhaps suboptimal) Left move in $T$ gives $$ o_L(H^{R_k}+T)\ge o_R(H^{R_k}+Y_k') > \nu(0).$$
Moreover, by (\ref{eq:important}), for all $m$, $o_L(H+X_m') > \nu(0)$.
 For the left most inequality, Right played a (perhaps suboptimal) move in the $T$ component.
\end{proof}

The next lemma concerns an elementary rephrasing of the downlinked idea for any universe of games.
\begin{lemma} \label{pre} Let $G,H\in \U\subseteq(\Omega,\A )$.
Then $G\succcurlyeq H$ implies $\forall H^L$, $G\not\!\swarrow \! H^L$, and $\forall G^R$, $G^R\not\!\swarrow \! H$.
\end{lemma}
\begin{proof}
For all $T\in \U$, for any $H^L$, we have
$$o_L(G+T)\geqslant o_L(H+T)\geqslant o_R(H^L+T).$$
Therefore, we cannot have $o_L(G+T)<o_R(H^L+T)$.

For all $T\in \U$, for any $G^R$, we have
$$o_L(G^R+T)\ge o_R(G+T)\geqslant o_R(H+T).$$
Therefore, we cannot have $o_L(G^R+T)<o_R(H+T)$.
\end{proof}

\section{The main result}\label{sec:main}
We summarize the results in the previous section by proving the Maintenance of the Absolute Fundamental Theorem, Theorem~\ref{thm:basic}. It says that, in any absolute universe, if Left prefers a game $G$ before $H$, then she has a  response to maintain this preference for any starting move by player Right, and where we think of a move in the $H$ component as a move on its conjugate, by the other player. 

\begin{theorem}[Maintenance] \label{thm:cnp}
Suppose that $\U\subseteq(\Omega,\A )$ is an absolute universe of combinatorial games and let $G,H\in \U$.
If $G\succcurlyeq H$, then:
\begin{enumerate}[]
  \item For all $G^R$, there is a $H^R$ such that $G^R\succcurlyeq H^R$ or there is a $G^{RL}$ such that $G^{RL}\succcurlyeq H$;
  \item For all $H^L$, there is a $G^L$ such that $G^L\succcurlyeq H^L$ or there is a $H^{LR}$ such that $G\succcurlyeq H^{LR}$.
\end{enumerate}
\end{theorem}

\begin{proof}
By Lemma~\ref{pre}, $G$ is not downlinked to any $H^L$ and no $G^R$ is downlinked to $H$.

Therefore Lemma~\ref{lem:main} gives that, for all $H^L$, there is a $G^L$ such that $G^L\succcurlyeq H^L$ or there is a $H^{LR}$ such that $G\succcurlyeq H^{LR}$ and for all $G^R$, there is a $H^R$ such that $G^R\succcurlyeq H^R$ or there is a $G^{RL}$ such that $G^{RL}\succcurlyeq H$.
\end{proof}

Next we restate and prove the main result, the Absolute Fundamental Theorem, Theorem~\ref{thm:basic}, of semi-local game comparison, also called basic order or play-comparison \cite{LarssNS2018,LarssNS2018c}, for absolute universes.\\

\noindent {\bf Theorem \ref{thm:basic}.} (Basic Game Comparison)
{\em Suppose $\U\subseteq(\Omega,\A )$ is an \abs universe of combinatorial games and let $G,H\in \U$. Then $G\succcurlyeq H$ if and only if the following two conditions hold\vspace{0.2 cm}

\noindent
Proviso:
\begin{enumerate}[]
 \item $o_L(G+X)\geqslant o_L(H+X)$ for all Left-atomic $X\in \U$;

 \item
 $o_R(G+X)\geqslant o_R(H+X)$ for all Right-atomic $X\in \U$;
\end{enumerate}
\noindent
Maintenance:
\begin{enumerate}[]
\item For all $G^R$, there is an $H^R$ such that $G^R\succcurlyeq H^R$, or there is a $G^{RL}$ such that $G^{RL}\succcurlyeq H$;
\item For all $H^L$, there is a $G^L$ such that $G^L\succcurlyeq H^L$, or there is an $H^{LR}$ such that $G\succcurlyeq H^{LR}$.
\end{enumerate}
}
\begin{proof}
($\Rightarrow$) $G\succcurlyeq H$ implies the Maintenance by Theorem \ref{thm:cnp}. The Proviso also holds because, if not, it directly contradicts $G\succcurlyeq H$.\\

($\Leftarrow$) Assume the Proviso and the Maintenance, and suppose that $G\not\succcurlyeq H$. The last happens because $o_L(G+X)<o_L(H+X)$ or $o_R(G+X)<o_R(H+X)$ for some $X$. Consider an $X$ of smallest birthday in those conditions and, without loss of generality, assume  $o_L(G+X)<o_L(H+X)$. Left starts disjunctive sum $H+X$. We have three cases:
\begin{enumerate}
  \item The game $X$ cannot be Left-atomic since this is in contradiction with the assumption of Proviso.
  \item Suppose that Left's optimal move is in the $H$ component, i.e. $o_L(H+X)=o_R(H^L+X)$. Because of the Maintenance, we have either, that there exists a $G^L$ such that $G^L\succcurlyeq H^L$ or there exists a $H^{LR}$, such that $G\succcurlyeq H^{LR}$. If the first holds, $$o_L(G+X)\geqslant o_R(G^L+X)\geqslant o_R(H^L+X)=o_L(H+X).$$ If the second holds, $$o_L(G+X)\geqslant o_L(H^{LR}+X)\geqslant o_R(H^L+X)=o_L(H+X).$$ Both contradict the assumption $o_L(G+X)<o_L(H+X)$.
  \item Suppose $o_L(H+X)=o_R(H+X^L)$. By the smallest birthday assumption, $o_R(G+X^L)\geqslant o_R(H+X^L)$. Therefore, $$o_L(G+X)\geqslant o_R(G+X^L)\geqslant o_R(H+X^L)=o_L(H+X).$$ Once again, this contradicts the assumption $o_L(G+X)<o_L(H+X)$.
\end{enumerate}
Hence, we have shown that $G\succcurlyeq H$.
\end{proof}

It may seem at a first glance that the `for all atomic $X$' part of the Proviso is still a burden. But, we will see in the next section, in Proposition ~\ref{thm:allcomps}, that the Proviso will often be much simplified, in a shift from basic to subordinate order.

\section{\Abs universes and their provisos}\label{sec:subordinate}
In Section~\ref{sec:npoe}, we continue investigating how normal-play is central to our theory, via an embedding of universes. Then, in Section~\ref{sec:subord}, we study how the basic game comparison, from the last section, translates to a given absolute universe, to obtain a subordinate and effective level of game comparison.

Recall that we compare games modulo some universe $\U$. In this section we will develop tools for interchanging comparison universes, while keeping the game trees identical. In this study we settle the problem: how does comparison in normal-play compare with comparison in other \abs universes? We sometimes index the set of adorns via $\A=\A_\U$ to clarify that the set of adorns belong to a given universe $\U\subseteq \Omega$. For comparison in a universe $\U$, we abbreviate $G\succcurlyeq H$ modulo $\U$, by $G\succcurlyeq_\U H$.

\subsection{A normal-play order embedding }\label{sec:npoe}
In normal-play, $X$ atomic implies that $X$ is a number.
By the Number Avoidance Theorem, \cite{Conwa2001}, `do not play in a number, unless all components are numbers',  the Proviso reduces to
``$o_L(G)\geqslant o_L(H)$ and $o_R(G)\geqslant o_R(H)$'', and this follows by induction on the
Maintenance. Hence, as we already noted in the Second Fundamental Theorem, in this case, the proviso can be removed.

Although intuitively clear, we have not yet formally introduced the concept of ``winning'', but of course, for example in normal-play, Left wins $G$ independently of who starts whenever $o(G) = (+1,+1)$.\footnote{An order preserving `winning function' can be defined to make the jump from outcomes to winning consistent with our theory.} The reason for this apparent negligence, is that the various abstract game comparisons do not require an explicit definition of winning. In normal- and mis\`ere play win-loss is the natural notion, but a scoring-play universe cannot in general be reduced to a na\"ive win-loss situation, although in a recreational playing the terminal score is often reinterpreted as win-draw-loss. In theory, it is more convenient to apply a minimax algorithm on the terminal scores. As we have seen, normal- and mis\`ere-play can also be conveniently recasted in this manner via the $\nu$ function.

Moreover, we may relax a game in any universe to view only its game tree, and then play it using the normal-play convention. Formally, we use a map on the group of adorns, where each adorn becomes the trivial element; given a game $G\in \U$, for each atom $\atom{a}$ in each atomic follower of $G$, let $\atom{a}$  $\rightarrow \atom{0}$. Denote this normal-play projection of $G\in \U$ by $G_{\Np}\in \Np$, and we may omit the subscript when the context is clear.

Further universal aspects of normal-play are revealed in the following observation and theorems.

\begin{observation}\label{obs:mup}
We revisit Example~\ref{ex:dicot}, which concerns dicot mis\`ere.
In both normal- and mis\`ere-play,
let $*=\cg{\bs 0} {\bs 0}$, $\uparrow \; = \{\bs 0\mid *\}$ and $\mup \;=  \{\bs 0,*\mid *\}$
be the games where these are the literal forms (no reductions by domination or reversibility)
There is an \emph{order-preserving} map but not an \emph{order-embedding}
of mis\`ere-play into normal-play. For instance, consider $\D$ the dicot games with mis\`{e}re-play convention. Then $\mup \;
\succcurlyeq_{\D} \bs 0$, and similarly $\mup\; =_{\Np} \; \uparrow \; \su_{\Np} \bs0$.
But the opposite is not true: $\uparrow\; \su_{\Np} \bs 0$, but $\uparrow \fuzzyU{\D}  \bs 0$.
\end{observation}

\begin{theorem}[Normal-play Order Preserving Map]\label{thm:npop}
Let $G, H\in \U$, an \abs universe. If $G\su_\U H$ then $G\su_{\Np} H$.
\end{theorem}

\begin{proof}
Apply the Maintenance for the universe $\U$ as a strategy for Left to win $G_\Np-H_\Np$ playing second.
\end{proof}

However $G\succ_\U H$ does not imply $G\succ_\Np H$; see Example~\ref{ex:dicot}.

We can prove a little more, by looking beyond Observation~\ref{obs:mup}, and instead adapt ideas from guaranteed scoring games \cite{LarssNNS2016,LarssNS2018,LarssNS2018c}.

Let $G\in \mathbb{Np}$, and let $\U\subseteq\Omega$ be a universe of combinatorial games.
Define the \emph{normal-play mapping} $\zeta:\mathbb{Np} \hookrightarrow \U$ as $\zeta(G)\equiv G$, where ${\bs 0}\in \A_\Np$ becomes instead ${\bs 0}\in \A_\U$, and as before ``$\equiv$'' denotes identical literal forms. The game tree and all adorns are identical, but play convention and thus the surrounding context changes, i.e. the interpretation of `$\boldsymbol 0$' is reflected in the maps $\nu_L$ and $\nu_R$ for the respective universes. Observe that in the game $\zeta(G)$, any terminal score is 0, and so played alone this game is trivial. But in disjunction with another scoring game, it could have a major influence on the outcome. For example, if $H\in \U$ is a zugzwang, then the normal-play game $G=\bs 1$, gives a beneficial situation for Left in the game $\zeta(G)+H$.

In normal-play if the game value is an integer $n$, it denotes ``$n$ moves for Left'' if $n \ge 0$ and otherwise it denotes ``$n$ moves for Right''. If we let $\U$ be a scoring-play universe, then $n$ normal-play moves map to $n$ scoring-play moves (for the same player), indicated by $\zeta(n) = \overline n$, if $n\ge 0$, and $\zeta(n) = \underline n$ if $n\le 0$, as not to confuse with the other type of `numbers' in scoring-play, that is the `scores' $\boldsymbol a = \pura{a}{a}$.

For a game $G$, let $${\rm adorn}_L(G) = \{\ell\mid \cgs{\atom{\ell}}{\GR} \text{ is a follower of } G\}$$ and $${\rm adorn}_R(G) = \{r\mid \cgs{\GL}{\atom{r}} \text{ is a follower of } G\}$$ denote sets of adorns in atomic followers of $G$, and let ${\rm adorn}(G) ={\rm adorn}_L(G)\cup {\rm adorn}_R(G)$.

\begin{definition}[Horror Vacui]\label{def:horvac}
A universe $\U\subseteq(\Omega,\A)$ is horror vacui, \footnote{``Nature abhors a vacuum" is a postulate attributed to Aristotle.} if, for all games of the forms $X=\cg{X^\mathcal{L}}{\emptyset^r}$, for $i\in \{L,R\}$, $$\nu_i(\max ({\rm adorn }_i(G))\le \nu_R (r),$$ and for all games of the form $X=\cg{\emptyset^\ell}{X^\mathcal{R}}$, for $i\in \{L,R\}$, $$\nu_i(\min ({\rm adorn}_i(G))\ge \nu_L (\ell).$$
\end{definition}

Intuitively, if a player is faced with a position in which they have no move, they would rather `pass' and have
the other player continue than have the game finish. Of course, in normal-play, the worst possible `option' is the
empty set. In mis\`{e}re-play the `worst' would be an infinite string of moves, but such games go beyond this study.
The \emph{horror vacui} property captures the worst
of normal-play, and, in such horror vacui universes, we will show that we have an order embedding to normal-play. By definition, each normal-play universe is horror vacui, but no mis\`ere-play universe is horror vacui. Guaranteed scoring-play is the archetype of a non-normal-play horror-vacui universe.

\begin{theorem}[Normal-play Order Embedding]\label{thm:npoe}
Suppose that $\U\subseteq(\Omega,\A)$ is an \abs and horror vacui universe. Then, the normal-play mapping $\zeta:\Np\hookrightarrow \U$ is
an order-embedding.
\end{theorem}

\begin{proof}
Theorem \ref{thm:npop} holds in particular for any game $G\in\U$, an absolute universe, where all adorns are 0s. Thus, it suffices to prove that $G\su_\Np H$
implies $\zeta(G) \succcurlyeq_\U \zeta(H)$.

First, suppose that Left has no option in $\zeta(H)+X$, $X\in \U$. This occurs when $X=\cg{\emptyset^x}{X^\mathcal{R}}$ and $H=\cg{\emptyset^0}{H^\mathcal{R}}$. In this case, $o_L(\zeta(H)+X) = \nu_L(x)$, with evaluation in the universe $\U$.
Now  $\zeta(G)+X = \zeta(G) + \cg{\emptyset^x}{X^\mathcal{R}}$; if Left has a move in $G$, then Right could never  improve his prospects by playing in the $X$ component, because $\U$ is horror vacui (Definition~\ref{def:horvac}). If Left cannot play in $G$ the game is over and
$o_L(\zeta(G)+X) = \nu_L(x)$. In both cases, $o_L(\zeta(G)+X) \geqslant \nu_L(x)=o_L(\zeta(H)+X)$.

Now assume that Left has a move in $\zeta(H)+X$, and with evaluation in $\U$, suppose first that there is a Left move, $X^{L}$, such that
  $o_L(\zeta(H)+X)=o_R(\zeta(H)+X^{L})$. Then, by induction, $o_R(\zeta(H)+X^{L})\leqslant o_R(\zeta(G)+X^{L})$ and
  $o_R(\zeta(G)+X^{ L})\leqslant o_L(\zeta(G)+X)$ which yields $o_L(\zeta(H)+X) \leqslant o_L(\zeta(G)+X)$.

  The remaining case is that there is a Left move, $H^L$, with $o_L(\zeta(H)+X)=o_R(\zeta(H^{L})+X)$.\footnote{We may identify $\zeta(H^L)=\zeta(H)^L$.} In $\mathbb{Np}$, $G\su H$, i.e., $G - H\su_\Np  \bs{0}$ and  so Left has a winning move in $G-H^{L}$. There are two possibilities, either $G^{L}-H^{L}\su_\Np \bs{0}$, for some $G^{L}$,
or $G-H^{LR}\su_\Np  \bs{0}$, for some $H^{LR}$. If the first occurs, then $G^{L}\su_\Np  H^{L}$
and, by induction, $o_R(\zeta(G^{L})+X)\geqslant o_R(\zeta(H^{L})+X)$ which gives the inequalities
\[o_L(\zeta(H)+X)=o_R(\zeta(H^{L})+X)\leqslant o_R(\zeta(G^{L})+X)\leqslant o_L(\zeta(G)+X).\]
If $G-H^{LR}\su_\Np  \bs{0}$ occurs, then, by induction,
$o_L(\zeta(G)+X) \geqslant o_L(\zeta(H^{LR})+X)$. We also have $o_L(\zeta(H^{LR})+X) \geqslant o_R(\zeta(H^{L})+X)$ and
since we are assuming that $o_L(\zeta(H)+X)=o_R(\zeta(H^{L})+X)$,
we can conclude that $o_L(\zeta(G)+X)\geqslant o_L(\zeta(H)+X)$.
\end{proof}

\subsection{Interpretation of the provisos in different universes}\label{sec:subord}
Absolute universes as considered in the literature, have unique interpretations of the Proviso, on the subordinate level. For each scoring-play universe, we have $\A=\mathbb R$.
\begin{enumerate}
\item In {\em normal-play}, $\Np$, the Proviso is implied by Maintenance.
\item In \textit{dicot mis\`{e}re-play}, $\D$, the cases $X$ Left-atomic or $X$ Right-atomic individually implies $X\equiv \bs{0}$. Consequently,
requiring that $o_L(G+X)\geqslant o_L(H+X)$ ($X$ Left-atomic) and $o_R(G+X)\geqslant o_R(H+X)$ ($X$ Right-atomic) is the same as requiring $o_L(G)\geqslant o_L(H)$ and $o_R(G)\geqslant o_R(H)$.

\item The case of {\em dead-ending mis\`ere-play}, $\E$, has some similarities with guaranteed scoring-play below, and we define the set of  waiting protected games, adapted from \cite{LarssMNRS2018}.\footnote{Waiting protected games are called `perfect murder' games in \cite{LarssMNRS2018}.} The Proviso simplifies to waiting protected  outcomes, defined by $\underline{o}_L(G) =  \min\{o_L(G+\underline n) : n\ge 0\}$ and
 $\overline{o}_R(G) = \max\{o_R(G+\overline{n}) : n\ge 0\}$. Where, $\underline{{\bs 0}}={\bs 0}$, and for $n\ge 1$,  $\underline n$ is the game where Right can either eliminate the game, or play to $\underline{n-1}$. The analogous concept holds for $\overline{n}$ and Left. We use that $\underline{o}_L(G)\geqslant \underline{o}_L(H)$ and $\overline{o}_R(G)\geqslant \overline{o}_R(H)$ together with the Maintenance imply the Proviso (by induction). Recall Example~\ref{ex:deadend}.

\item The structure of the {\em free mis\`{e}re space}, $\M$, is not as rich as that of $\mathbb{Np}$.
For example, $G=\bs{0}$ implies that $G\equiv \bs{0}$, the empty tree. Suppose that $G^\mathcal{L}\neq\atom{0}$. Then, by a mimic-strategy, Left loses playing first in the game $G+\{\atom{0}\,|\,{G^\mathcal{L}}^\circ\}$, where ${G^\mathcal{L}}^\circ$ is the set of all adjoints of $G^L\in \GL$. This proves that $G\not\equiv {\bs 0}$, a next player winning position. For the definition of adjoint, see Definition~\ref{def:adjoint} adapted from  \cite{Siege2013} Chap. V.6.). Siegel \cite{Siege2015b} proved that we can still have local comparison.

\item For \textit{dicot scoring-play}, $\mathbb {Ds}$, Ettinger \cite{Ettin1996} proved local comparison for scores, i.e. ``Is $G\su \pura{a}{a}$''.
The cases $X$ Left-atomic and $X$ Right-atomic imply $X=\pura{a}{a}$ (or more generally $X=\pura{\ell}{r}$) and hence
the Proviso only needs to consider alternating play, similar to dicot mis\`ere.

\item \textit{Guaranteed scoring-play} \cite{LarssNNS2016}, $\mathbb{Gs}$, has a straightforward translation of the Proviso, somewhat simpler than for dead-end mis\`ere,
and a normal-play order embedding. The last is explained by the fact that the `guaranteed property' \emph{is} the horror vacui property in Definition~\ref{def:horvac}.
The Proviso simplifies to the \emph{waiting-protected left- and right-scores}, defined by $\underline{o}_L(G) =  \min\{o_L(G+\underline n) : n\ge 0\}$ and
 $\overline{o}_R(G) = \max\{o_R(G+\overline{n}) : n\ge 0\}$.\footnote{Waiting protected is called pass-allowed in \cite{LarssNNS2016}.} Where, $\underline{0}={\bs 0}$, and for $n\ge 1$,  $\underline n$ is the game where Right can play to $\underline{n-1}$. That is, $\underline{o}_L(G)$ is the best score Left can achieve if Right, and only Right,
is allowed to pass as many times as he wishes (as discussed in Section~\ref{sec:guaranteed}).
Similarly, $\overline{o}_R(G)$ is the best score that Right can achieve if Left is allowed to pass. We use that $\underline{o}_L(G)\geqslant \underline{o}_L(H)$ and $\overline{o}_R(G)\geqslant \overline{o}_R(H)$ together with the Maintenance imply the Proviso (by induction). Recall Example~\ref{ex:guaranteed}.
\item \textit{Free scoring space} \cite{Stewa2011, Stewa2019}, $\mathbb S$. The situation is similar to free mis\`ere-play, but with infinitely many terminal scores. A similar argument to that in item 3 gives that $G=\pura{\ell}{r}$ if and only if $G\equiv \pura{\ell}{r}$.
\end{enumerate}
We get the following result, concerning the details of a local (or subordinate) Proviso, of the aforementioned absolute universes. Note that the pass allowed outcomes depend on the given universe, dead-end or scoring, respectively.
\begin{theorem}[Subordinate Game Comparison]\label{thm:allcomps}  Let $G, H\in \U\subseteq(\Omega,\A)$, an \abs universe. Then $G\succcurlyeq_{\U} H$ if the Maintenance holds and ${\U} $ is
\begin{enumerate}
\item  normal-play;
\item dicot mis\`ere-play, and $o(G)\geqslant o(H)$;
\item dead-ending mis\`ere-play, and $\underline{o}_L(G)\geqslant \underline{o}_L(H)$ and $\overline{o}_R(G)\geqslant \overline{o}_R(H)$;
\item free mis\`ere-play, and $H^\mathcal{L}=\atom{0}\Rightarrow G^\mathcal{L}=\atom{0}$ and $G^\mathcal{R}=\atom{0}\Rightarrow H^\mathcal{R}=\atom{0}$;
\item dicot scoring-play, and $o(G)\geqslant o(H)$;
\item guaranteed scoring-play, and $\underline{o}_L(G)\geqslant \underline{o}_L(H)$ and $\overline{o}_R(G)\geqslant \overline{o}_R(H)$;
\item free scoring-play, and $H^\mathcal{L}=\atom{a}\Rightarrow G^\mathcal{L}=\atom{b}$ with $b\ge a$, and $G^\mathcal{R}=\atom{a}\Rightarrow H^\mathcal{R}=\atom{b}$ with $a\ge b$.
\end{enumerate}

\end{theorem}
\begin{proof}
All, except item 7 is explained in the list preceding the theorem. Let us give the proof of item 7. We are assuming Maintenance and wish to prove that: $H^\mathcal{L}=\atom{a}\Rightarrow G^\mathcal{L}=\atom{b}$ with $b\ge a$, and $G^\mathcal{R}=\atom{a}\Rightarrow H^\mathcal{R}=\atom{b}$ with $a\ge b$, is equivalent with
\begin{enumerate}
 \item $o_L(G+X)\geqslant o_L(H+X)$ for all Left-atomic $X\in \Omega$;
 \item $o_R(G+X)\geqslant o_R(H+X)$ for all Right-atomic $X\in \Omega$;
\end{enumerate}
\noindent $(\Rightarrow)$ Assume $H^\mathcal{L}=\atom{a}$ and $G^\mathcal{L}=\atom{b}$,  with $b\ge a$. Then Left cannot move in both $G+X$ and $H+X$, where $X$ is Left atomic. By the assumption $b\ge a$, item 1 holds by additivity of adorns. Analogously item 2 holds.\\
\noindent $(\Leftarrow)$ Given the Proviso, i.e. items 1 and 2, and Maintenance, we wish to show that $H^\mathcal{L}=\atom{a}\Rightarrow G^\mathcal{L}=\atom{b}$, with $b\ge a$. Assume $H^\mathcal{L}=\atom{a}$. Thus, $o_L(H+X)= a+x$, if $X$  Left-atomic, with $\XL=\atom{x}$. Suppose that $|\GL| > 0$. Then, we must find a Left-atomic $X$ such that $o_L(G+X) < a+x$, where $\XL=\atom{x}$. Let $X=\cgs{\atom{x}}{\,\conj{\GL}}$, with $x=1-a$, and so $o_L(G+X) \le 0 < a+x$. Therefore $|\GL|= 0$, and the proviso implies $b\ge a$.
\end{proof}

\section{Discussion of outcome saturation}\label{sec:dense}
Consider a universe $\U$ and a game $G\in \U$. Recall the definition of conjugate and adjoint.
The conjugate of $G$ is
\[ \conjt{G}\, =
\begin{cases}
\pura{-r}{-\ell}, &\mbox{ if $G=\pura{\ell}{r}$, $\ell, r\in \A$}\\
\cgs{\conjt\GR\;}{\emptyset^{-\ell} }, &\mbox{ if $G=\cgs{\emptyset^{\ell}}{\GR}$}\\
\cgs{\emptyset^{-r}}{\;\conjt\GL }, &\mbox{ if $G=\cgs{\GL}{\emptyset^{r}}$}\\
\cgs{\conjt\GR\; }{\;\conjt\GL }, & \mbox{ otherwise}.
\end{cases}
\]

 The adjoint of $G$ is
\[
G^\circ =
\begin{cases}
\cg{-{\bs \ell}}{-{\bs r}} & \textrm{if  $G = \pura{\ell}{r}$}; \\
\cg{\GR^\circ}{-{\bs \ell}} & \textrm{if $|\GR|>0$ and $\GL=\atom{\ell}$}; \\
\cg{-{\bs r}}{\GL^\circ} & \textrm{if $|\GL|>0$ and $\GR=\atom{r}$}; \\
\cg{\GR^\circ}{ \GL^\circ} & \textrm{otherwise}.
\end{cases}
\]

The adjoint has an analogous property to the conjugate. By using the evaluation function, $\nu$, we may unify the classical settings of normal- and mis\`ere-play with that of scoring-play etc. Typically, the adjoint is used in mis\`ere-play, whereas the conjugate is used in normal-play. As we will see, either may be useful in scoring-play. The following lemma is used implicitly in the proof of Theorem~\ref{thm:pardense}.
\begin{lemma}\label{lem:adjoint}
Consider a universe $\U$ and a game $G\in \U$. Then $o_L(G+G^\circ)\le \nu_R(0)$ and
$o_L(G+\conj{G}\,)\le \nu_L(0)$.
\end{lemma}
\begin{proof}
Right can mimic each Left move. Then use additivity of the adorns. In the first case, Right is the current player at the terminal position. In the second case, this player is Left.
\end{proof}

For any scoring universe $\U$, with $G\in\U$ and $\ell,r\in \R$, we have $o_L(G+G^\circ+\pura{\ell}{r})\le \nu(r)$ and $o_L(G+\conjt{G}+\pura{\ell}{r})\le \nu(\ell)$.

\begin{observation}\label{obs:conjadj}
For any scoring universe $\U$, with $G\in\U$ and $\ell,r\in \R=\A$, we have $o_L(G+G^\circ+\pura{\ell}{r})\le r$ and $o_L(G+\conjt{G}+\pura{\ell}{r})\le \ell$. Similarly, $o_R(G+G^\circ+\pura{\ell}{r})\ge \ell$ and $o_R(G+\conjt{G}+\pura{\ell}{r})\ge r$.
\end{observation}
The inequalities can be reversed, if we make a small change to the `adjoint', which is the variant we use in the proof of the main lemma in this paper. This is the variant that induces the definition of (scoring-play) saturation.\footnote{For the classical winning conditions (normal- and mis\`ere-play) saturation is usually the stronger notion: for any game $G$, and any outcome $x$, there is a game $H$ such that $o(G+H)=x$.}
\begin{observation}\label{obs:conjadj2}
For any scoring universe $\U$, for any $G\in\U$, $x\in \R$, we have $o_L(G+G^\circ+\cg{\bs \ell}{\cdot})\ge \bs \ell$ and $o_L(G+\conjt{G}+\cg{\bs \ell}{\cdot})\ge \bs \ell$. Similarly $o_R(G+G^\circ+\cg{\cdot}{\bs r})\le \bs r$ and $o_R(G+\conjt{G}+\cg{\cdot}{\bs r})\le \bs r$.
\end{observation}

We review, the mis\`ere-play outcome saturation, and normal-play is analogous, but instead of the adjoint it uses the conjugate.

\begin{definition}[Mis\`ere-play Adjoint \cite{Siege2013}]\label{def:adjoint}
Let $\U$ be a mis\`ere-play universe. The \emph{adjoint} of $G\in \U$, is the game $G^\circ$, given by
\[
G^\circ =
\begin{cases}
\cg{{\bs 0}}{{\bs 0}} & \textrm{if  $G = {\bs 0}$}; \\
\cg{\GR^\circ}{{\bs 0}} & \textrm{if $G \neq {\bs 0}$ and $G$ is Left-atomic}; \\
\cg{{\bs 0}}{\GL^\circ} & \textrm{if $G \neq {\bs 0}$ and $G$ is Right-atomic}; \\
\cg{\GR^\circ}{ \GL^\circ} & \textrm{otherwise},
\end{cases}
\]
where $\GR^\circ$ is the set of games of the form ${G^R}^\circ$.
\end{definition}

Note that the adjoint $G^\circ\in \U$, if $G\in \U$, and it is never atomic. The following observation is well known.

\begin{observation}[Mis\`ere-play Saturation]\label{obs:misout}
Given any game $G$ in mis\`ere-play, using the adjoint $G^\circ$, we get that $G+G^\circ$ is a P-position, that is $o(G+G^\circ)=(-1,+1)$.\footnote{Here, we use the classical notation of outcomes relating to the standard winning convention.}

Given any mis\`ere-play game $G$, it is required to find games $G_L, G_R, G_N$ and $G_P$ such that each outcome class is represented, that is such that $o(G+G_i)=i$. We already shown this for $G_P = G^\circ$. For  $G_N$, we take $G_N=\cg{G^\circ}{G^\circ}$; the first player to move in $G + G_N$ can move to $G+G^\circ$ and win. Hence $o(G+G_N)=\mathscr N$. Consider the game $G_L=\cg{\GR^\circ, G^\circ}{G_N}$ (where we assume there is a $G^R$ otherwise omit that part). We determine the outcome of $G+\cg{\GR^\circ, G^\circ}{G_N}$. If Left begins she can play to $G + G^\circ$ and win. If Right begins by playing in the $G$-component,  then Left can respond to $G^R+{G^R}^\circ$ and win because this is a P-position. Otherwise Right plays to $G+G_N$ which is an (already proved) N-position, so Left wins in either case. Hence $o(G+G_L)=\mathscr L$, and an analogous argument proves that $o(G+G_R)=\mathscr R$.
\end{observation}

\section{Future directions and open problems}
Let us mention a couple of related problems.

\begin{itemize}
\item Are there infinitely many absolute universes \cite{LarssNSabsinf}?
\item Prove reduction theorems, via domination and reversibility, and discuss the possibility of an absolute `canonical form' game value. For what choice of `smallest' do we obtain uniqueness? Observe that for normal-play no choice is required, because uniqueness is automatic.  For the Guaranteed and Dead-ending universes a choice of `smallest' form is necessary.
\item Normal-play games satisfy a nice categorical structure. It is demonstrated in \cite{LarssNS2018} that certain aspects of this structure generalize to absolute universes. In that paper it is also shown that Absolute game comparison has an analogue play-comparison to normal-play games, but where the given Proviso induces the termination of a game. See \cite{LarssNS2018} for open problems in related areas.
\item Prove mean value theorems that generalize the classical settings of Hanner (with respect to `scores') and Conway (with respect to `moves'). In particular, is it possible to have a combined mean value with respect to `scores' and `moves' in various scoring-play settings, say for example in Guaranteed scoring-play?
\item Is there an absolute theory for Loopy (cyclic) games?
\end{itemize}

\bibliographystyle{plain}
\bibliography{games7.bib}
\end{document}